\documentclass[reqno]{amsart}

\usepackage{mymacro}
\usepackage{bookmark}
\usepackage{fancyhdr}
\usepackage{stmaryrd}

\setlist[enumerate, 1]{label = \roman*., ref = \roman*}
\setlist[enumerate, 2]{label = \theenumi.\alph*}

\newcommand{\Hphi}{\mathsf{H}(\varphi)}
\newcommand{\Tphi}{\mathsf{T}(\varphi)}

\makeatletter
\@namedef{subjclassname@2020}{%
  \textup{2020} Mathematics Subject Classification}
\makeatother

\author[Joaqu\'in Brum]{Joaqu\'in Brum}
\address{\hskip-\parindent{}
	IMERL, Facultad de Ingeniería, Universidad de la República, Julio Herrera y Reissig 565, Montevideo, Uruguay.}
\email{joaquinbrum@fing.edu.uy}

\author[Mart\'in Gilabert Vio]{Mart\'in Gilabert Vio}
\address{\hskip-\parindent{}
	Institut Camille Jordan, Université Claude Bernard Lyon 1, 43 boulevard du 11 novembre
1918, 69622 Villeurbanne, France.}
\email{gilabert@math.univ-lyon1.fr // martingilabertvio@gmail.com}

\author[Nicol\'as Matte Bon]{Nicol\'as Matte Bon}
\address{\hskip-\parindent{}
	CNRS \& Institut Camille Jordan, Université Claude Bernard Lyon 1, 43 boulevard du 11 novembre
1918, 69622 Villeurbanne, France.}
\email{mattebon@math.univ-lyon1.fr}

\date{\today}

\title{Groups with classifiable actions on the line}

\begin{document}
	
\begin{abstract}
We motivate and study the class $\cC$ of countable groups $G$ such that the conjugacy relation between minimal actions of $G$ on $\R$ by orientation-preserving homeomorphisms is smooth --- that is, admits a Borel transversal. No example of amenable group outside of $\cC$ is known.   We show a number of stability properties of $\cC$ under group-theoretic operations and that $\cC$ contains all finitely generated groups of piecewise affine homeomorphisms of the interval. We exhibit a finitely generated group $G$ that is not in $\cC$, such that $G$ is amenable if and only if Thompson's group $F$ is amenable. We also prove that the semiconjugacy relation among cocompact actions of a countable group $G$ is smooth if and only if $G \in \cC$, and that it is essentially countable even when $G$ is not finitely generated. In the Appendix, we show that there is no good analogue of the space of harmonic actions for a countable non-finitely generated group. 
\end{abstract}
	
\maketitle

\section{Introduction} \label{section:introduction}

\subsection{Context} One of the central goals in dynamical systems is to identify explicit invariants that classify systems in a given class up to conjugacy.  The dynamical systems we consider here are group actions by orientation-preserving homeomorphisms on connected one-manifolds, i.e. the interval and the circle. 

The situation on the circle is somewhat better-behaved, at least modulo actions that have a finite orbit (whose study essentually reduces to study actions on the line).  First, an action $\varphi \colon G \to \mathrm{Homeo}_0(S^1)$ always admits a unique \emph{minimal set} $\Lambda$, that is, a non-empty closed $\varphi$-invariant subset $\Lambda$ of $S^1$ such that action of $\varphi(G)$ on $\Lambda$ has only dense orbits, and after applying a semiconjugacy one can  reduce to the case where $\Lambda = S^1$, i.e. when $\varphi$ is minimal. Second, a theorem of É. Ghys \cite{Ghys1987} shows that two minimal actions on the circle are conjugate if and only if they have the same bounded Euler class. S. Matsumoto \cite{Matsumoto1986} gives a more elementary interpretation of this result.

 The existence of an explicit invariant that classifies a class of actions up to conjugacy --- regardless of the invariant --- is a result in its own right that can be formalized in the framework of Borel equivalence relations. Here we are thus led to consider the Polish space $\mathrm{Rep}_\mathrm{min}(G, S^1)$ of minimal actions of $G$ on $S^1$ and the action of the group $\mathrm{Homeo}_0(S^1)$ of orientation-preserving homeomorphisms of $S^1$ on $\mathrm{Rep}_\mathrm{min}(G,S^1)$ by conjugation. 
It is then apparent from \cite{Matsumoto1986} that there is a Borel map from $\mathrm{Rep}_\mathrm{min}(G, S^1)$ to a standard Borel space such that its fibers are exactly the conjugacy classes. The conjugacy relation on $\mathrm{Rep}_\mathrm{min}(G, S^1)$ is thus said to be \emph{smooth}.

The case of the real line is more complicated and is the subject of this article. Firstly, even when an action by orientation-preserving homeomorphisms of $\R$ has no global fixed points it may fail to admit a minimal set. This phenomenon cannot happen when $G$ is finitely generated, for which we can still reduce to minimal actions (up to a semi-conjugacy). But, secondly, even  in this case there may be no Borel map from the space $\mathrm{Rep}_\mathrm{min}(G, \R)$ of minimal actions  $\varphi\colon G\to \mathrm{Homeo}_0(\R)$ to a standard Borel space whose fibers are the conjugacy classes. For instance, the free group on two generators admits no such map (see Subsection \ref{subsection:introBorel} below), and this rules out an analogue of the Ghys--Matsumoto result. 

On the other hand, there are several classes of groups whose minimal actions can be explicitly classified. A foundational result of O. Hölder \cite{Holder1901} implies that all minimal actions of countable abelian groups on the real line are conjugate to actions by a dense group of translations. This remains true for all groups containing no non-abelian free semigroups, due to work of J. Plante \cite{Plante1975}. For another example see \cite{Rivas2010}, where C. Rivas proves that every solvable Baumslag-Solitar group $\mathrm{BS}(1,n), n \in \N_+$ admits only two minimal actions on $\R$ up to conjugacy.

Motivated by this discussion, in this paper we study the following class of groups. As mentioned above, we denote by $\mathrm{Rep}_\mathrm{min}(G, \R)$ the space of all minimal actions $\varphi\colon G\to \mathrm{Homeo}_0(\R)$. It is a Polish space with respect to pointwise convergence in the compact-open topology of $\mathrm{Homeo}_0(\R)$ (see Subsection \ref{subsection:spaces}), and hence a standard Borel space. 

\begin{defn}
	Define $\cC$ as the class of countable groups $G$ such that the conjugacy relation on $\mathrm{Rep}_\mathrm{min}(G, \R)$ is smooth.
\end{defn}

Hence a group $G$ is in $\cC$ if and only if there exists a Borel function $\mathrm{e}(\cdot)$ on $\mathrm{Rep}_\mathrm{min}(G, \R)$, with values in a standard Borel space, such that two minimal actions  $\varphi_1,\varphi_2$ are topologically conjugate if and only if $\mathrm{e}(\varphi_1)=\mathrm{e}(\varphi_2)$. A more concrete (and checkable) characterisation is provided in Theorem \ref{cor:criterio} below.

It follows from the preceding examples that all groups $G$ not containing a free semigroup on two generators (in particular, all abelian groups) belong to $\cC$. The class $\cC$ is also known to contain various examples of groups of exponential growth such as $\mathrm{BS}(1,n), n \in \N_+$. The following questions arise naturally.

\begin{qnnum} \label{qn:amenable} \ 
\begin{enumerate}
    \item \label{i-amenable} Is every finitely generated amenable group in $\cC$?
	\item  \label{i-elementary-amenable} Is every finitely generated elementary amenable group in $\cC$?
\end{enumerate}
\end{qnnum}
We do not know the answer to either question. A positive answer to Question \ref{qn:amenable}~\eqref{i-amenable} would give a strengthening of Witte-Morris' theorem that a finitely generated amenable group acting non-trivially on the line has a non-trivial homomorphism to $(\R, +)$ \cite{WitteMorris2006}. This connection becomes transparent using the \emph{space of harmonic actions} of Deroin, Kleptsyn, Navas and Parwani \cite{DKNP}: we explain it in Subsection \ref{subsection:deroinborel}.

\subsection{Main results} Our first main result is a stability property of the class $\cC$, that implies that many groups are in $\cC$. Recall that a subgroup $H\le G$ is \emph{commensurated} if $gHg^{-1}\cap H$ has finite index in $H$ for every $g\in G$. We denote by $\langle\langle H \rangle \rangle$ its normal closure.

\begin{thmA} \label{teo:stability}
	Let $G$ be a finitely generated group. If there is a commensurated subgroup $H \le G$ such that $H$ and $G/ \langle \langle H \rangle \rangle$ belong to $\cC$, then $G \in \cC$.
\end{thmA}

\begin{cor} \label{cor:stability}
	    \ 
\begin{enumerate}

	\item \label{i-stability-extension} If $G$ is a finitely generated group  having a normal subgroup $N \unlhd G$ such that $N$ and $G/ N$ belong to $\cC$, then $G \in \cC$. 
	\item \label{i-stability-wreath} If $H, K\in \cC$ are finitely generated, then the wreath product $H\wr K$ is in $\cC$.
	\item \label{i-stability-HNN} If $\Phi \colon H \to K$ is an isomorphism between finite-index subgroups $H,K$ of  a finitely generated group $G\in \cC$, then the HNN-extension $G^\Phi$ is in $\cC$.
	\item \label{i-stability-amalgamated} If $H, G \in \cC$ are finitely generated groups containing $K$ as a finite-index subgroup, then the amalgamated product $G \ast_K H$ is in $\cC$.
\end{enumerate}
	\end{cor}

\begin{ejm}
From item \eqref{i-stability-extension} it follows by induction that if $G$ is a finitely generated group having a series of normal subgroups
\[\{e\}=G_0\unlhd \cdots \unlhd G_m=G\]
such that $G_{i}/G_{i-1}\in \cC$ for $i=1,\ldots, m$ (e.g. if $G_{i}/G_{i-1}$ does not contain a free semigroup on two generators), then $G\in \cC$. In particular, $\cC$ contains all finitely generated virtually solvable groups: this could also be deduced from the results on actions of solvable groups on the line in \cite[Theorem B]{BrumMatteBonRivasTriestino2025}, but not the stability of $\cC$ under extensions. 

\end{ejm}
In view of Corollary \ref{cor:stability} \eqref{i-stability-extension}, an affirmative answer to Question \ref{qn:amenable} \eqref{i-elementary-amenable} would follow if $\cC$ were also stable under direct limits. Unfortunately this is not true: we provide a counterexample in Proposition \ref{prop:nolimites} (which is however not elementary amenable). It is nevertheless true that an infinite \emph{direct sum} of finitely generated groups in $\cC$ is still in $\cC$ (Proposition \ref{prop:producto}), which gives the stability under wreath product in item \eqref{i-stability-wreath}. This is a rich source of examples in view of the fact that the wreath product $H\wr K$ of any two countable groups acting faithfully on $\R$ always has \emph{minimal} faithful actions on $\R$ (see \cite[Example 8.1.8]{BMRT}).

Items \eqref{i-stability-HNN} 
and \eqref{i-stability-amalgamated} in Corollary \ref{cor:stability} are both special cases of a statement about fundamental groups of graphs of groups (Corollary \ref{cor:BassSerre}), which in turn follows from Theorem \ref{teo:stability}.

\begin{ejm}
Item \eqref{i-stability-HNN} implies that all Baumslag-Solitar groups $\mathrm{BS}(m,n) = \langle a,b \mid ab^m a^{-1} = b^n \rangle$ for $m,n \in \Z \setminus \{0\}$ belong to $\cC$. When $n, m \ge 1$ groups are known to admit faithful minimal actions on $\R$ by a construction of Farb and Franks \cite{FarbFranks2020}.  In  Subsection \ref{subsection:BaumslagSolitar} we observe that they admit an uncountable family of non-conjugate minimal actions when $2 \leq m <n$ (in contrast with the solvable case $\mathrm{BS}(1, n)$ \cite{Rivas2010}).
\end{ejm}

From \cite[Section 16.3.1]{BMRT} it was already known that Thompson's group $F$ of piecewise dyadically affine homeomorphisms of the interval belongs to $\cC$. We extend this result to any finitely generated group $G$ which can be embedded in the group $\mathrm{PProj}_0(\R)$ of orientation-preserving piecewise projective homeomorphisms of the line, with no assumptions on the embedding. Here, a homeomorphism $f \in \mathrm{Homeo}_0(\R)$ belongs to $\mathrm{PProj}_0(\R)$ if, outside a finite subset of $\R$, it is locally of the form $x \mapsto (ax + b)/(cx + d)$ where $a,b,c,d \in \R$ and $ad - bc = 1$.

\begin{thmA} \label{teo:PL}
	Let $G \subseteq \mathrm{PProj}_0(\R)$ be finitely generated. Then $G \in \cC$. In particular, any finitely generated group of piecewise affine homeomorphisms of $\R$ belongs to $\cC$.
\end{thmA}

Groups of piecewise affine homeomorphisms are a quite rich source of examples of (non-virtually solvable) elementary amenable groups acting on the line, by work of M. Brin \cite{Brin2005} (see also \cite{Navas2004}) and of Bleak-Brin-Moore \cite{BleakBrinMoore2021}. These groups all belong to $\cC$ by Theorem \ref{teo:PL}, giving some encouraging evidence towards an affirmative answer to Question \ref{qn:amenable}~\eqref{i-elementary-amenable}. 

The amenability of groups of piecewise affine homeomorphisms of $\R$ is an outstanding open problem, the most well-known special case being the amenability of Thompson's group $F$ (in contrast Monod showed that there are non-amenable groups of piecewise projective homeomorphisms \cite{Monod2013}). Theorem \ref{teo:PL} shows that they belong to $\cC$ in any case. This might give hope that Question \ref{qn:amenable}~\eqref{i-amenable} could be answer positively without having to decide the amenability of $F$. However, this is not the case:

\begin{thmA} \label{teo:noc}
	There is a finitely generated group $G \not \in \cC$, such that $G$ is amenable if and only if Thompson's group $F$ is amenable.
\end{thmA}

In particular, an affirmative answer to Question \ref{qn:amenable}~\eqref{i-amenable} would imply the non-amenability of $F$, and thus might be extremely hard to decide. It is still possible (and perhaps more likely) that a negative answer can be obtained without deciding the amenability of $F$: nevertheless this seems to require to find new sources of examples of amenable groups acting on the line. In either case, we believe that Question \ref{qn:amenable}~\eqref{i-amenable} is interesting.

\subsection{Beyond the class $\cC$: the Borel complexity of conjugacy}

When leaving the class $\cC$, one can still try to rank groups according to the richness of their space of minimal actions on the line. This problem can be framed within the theory of Borel reducibility of equivalence relations, which allows to compare different classification problems in mathematics formulated as Borel or analytic equivalence relations on standard Borel spaces. We refer to \cite{Hjorth2000, KB} for a presentation of the subject.

Given equivalence relations $E_1, E_2$ defined on standard Borel spaces $Z_1, Z_2$ respectively, we say that \emph{$E_1$ is reducible to $E_2$} if there exists a Borel map $r \colon Z_1 \to Z_2$ such that for all $x,y \in Z_1$, we have $(x,y) \in E_1$ if and only if $(r(x), r(y)) \in E_2$. This notion expresses in a rigorous way the idea that deciding if elements of $Z_2$ are $E_2$-related is at least as hard as deciding if elements of $Z_1$ are $E_1$-related. We say two relations are \emph{bireducible} if one is reducible to the other and vice-versa. A Borel equivalence relation is \emph{smooth} if it is reducible to the identity on some standard Borel space, so such relations are the simplest ones from this point of view.

The largest natural space of actions of a countable group $G$ on $\R$ that one may expect to understand is the space $\mathrm{Rep}_\mathrm{irr}(G)$ of actions with no global fixed point on $\R$, called \emph{irreducible} actions. A classical construction by A. Denjoy \cite{Denjoy1932} by blowing up orbits of an irreducible action $\varphi$ allows to produce many actions with roughly the same dynamics as $\varphi$ but not conjugate to it. The equivalence relation generated by the blow-up procedure coincides with semiconjugacy: two irreducible actions $\varphi_1, \varphi_2$ are \emph{semiconjugate} if there exists a non-decreasing (and not necessarily continuous) map $h \colon \R \to \R$ that intertwines $\varphi_1$ with $\varphi_2$. Semiconjugacy among irreducible actions is a natural equivalence relation in this context, and it boils down to conjugacy when both actions are minimal.

Write $\R \curvearrowright^\Psi \mathrm{Rep}_\mathrm{irr}(G)$ for the \emph{translation flow} conjugating an irreducible action by translations. Following \cite{DKNP, deroinAlmost}, if $G$ is finitely generated the study of symmetric random walks on $G$ allows one to construct a compact $\Psi$-invariant subspace $\mathrm{Harm}(G) \subseteq \mathrm{Rep}_\mathrm{irr}(G)$, the \emph{space of harmonic actions} of $G$, composed of minimal actions and actions of $G$ by translations together with a continuous retraction map $r \colon \mathrm{Rep}_\mathrm{irr}(G) \to \mathrm{Harm}(G)$ that reduces semiconjugacy on $\mathrm{Rep}_\mathrm{irr}(G)$ to the orbit equivalence relation of $\Psi$ on $\mathrm{Harm}(G)$. It was noticed in \cite[Section 14.4]{BMRT} that the existence of these objects, along with a theorem of V. Wagh \cite{Wagh1988} shows the following.

\begin{prop}[{\cite[Section 14.4]{BMRT}}]
	Let $G$ be a finitely generated group. The semiconjugacy relation between irreducible actions of $G$ on the line is essentially hyperfinite, and it is smooth if and only if $G \in \cC$.
\end{prop}

Here, a relation is \emph{essentially hyperfinite} if it is bireducible to the orbit equivalence relation of a Borel automorphism on a standard Borel space. These form a strictly larger class of relations than the smooth ones, albeit one of low complexity (see Subsection \ref{subsection:introBorel} below).

An irreducible action $\varphi \in \mathrm{Rep}_\mathrm{irr}(G)$ is said to be \emph{cocompact} if there is a compact subset of $\R$ intersecting all $\R$-orbits. Equivalently, the action $\varphi$ admits a minimal set (see \cite[Proposition 2.1.2]{Navas2011}). 

In Appendix \ref{subsection:deroinFinf}, we show that when $G$ is not finitely generated, there is in general no full analogue of the space $\mathrm{Harm}(G)$ that controls all cocompact actions of $G$ in a similar way as in the finitely generated case. We are able to show nonetheless that semiconjugacy among cocompact actions is \emph{essentially countable}, that is, bireducible to a Borel equivalence relation where every equivalence class is countable. In particular, it is a Borel equivalence relation, which is not evident \emph{a priori}.
\begin{thmA} \label{teo:complexity}
	Let $G$ be a countable group. The semiconjugacy relation between cocompact actions of $G$ on the line is essentially countable, and it is smooth if and only if $G \in \cC$.
\end{thmA}

\begin{qnnum}
Let $G$ be a countable group. Is it true that the semiconjugacy relation between cocompact actions of $G$ on the line  is always essentially hyperfinite?
\end{qnnum}

To the best of our knowledge these questions are unrelated to the study of the Borel complexity of a countable group acting on its space of left orders, which has been undertaken in \cite{CalderoniClay2024}.

\subsection{On proofs} The main criterion we use to verify membership in $\cC$ is Theorem \ref{cor:criterio}, which asserts that a group $G$ is in $\cC$ if and only if for every sequence $(f_n)_{n \geq 0} \subseteq \mathrm{Homeo}_0(\R)$ and every proximal minimal action $\varphi$ of $G$ such that $\lim_{n \to \infty} f_n.\varphi = \varphi$ we have $\lim_{n \to \infty} f_n = \mathrm{id}_\R$ (that is, any sequence of homeomorphisms almost centralizing $\varphi$ is asymptotically trivial). Equipped with this criterion, the proof of Theorem \ref{teo:stability} relies on the structural theory of group actions on the line and the notion of laminar actions introduced in \cite[Chapter 8]{BMRT} in a fundamental way. The proof of Theorem \ref{teo:PL} also involves structure theory for actions of micro-supported groups from \cite[Chapter 9]{BMRT}.

Even though our main results deal with finitely generated groups, we do not formulate the definition of the class $\cC$ in terms of harmonic spaces $\mathrm{Harm}(G)$. We do so since our definition of $\cC$ is natural, and to prove closure properties of $\cC$ it is crucial to consider groups that are not finitely generated (see also Appendix \ref{subsection:deroinFinf}).

\subsection{Organization of the article} Section \ref{section:preliminaries} collects the necessary preliminaries to read this paper. In Section \ref{section:spaces} we prove Theorem \ref{cor:criterio}, which is the main criterion we use to decide membership of a group in $\cC$.  Section \ref{section:actionsCommensurated} records some restrictions on groups acting minimally coming from commensurated subgroups, to be used in the proof of Theorem \ref{teo:stability}. Sections \ref{section:stability} proves Theorem \ref{teo:stability} and Corollary \ref{cor:stability}. Sections \ref{subsection:micro}, \ref{section:noC} and \ref{section:complexity} prove Theorems \ref{teo:PL}, \ref{teo:noc} and \ref{teo:complexity} respectively.
Finally, in Appendix \ref{subsection:deroinFinf} we show that it is not possible to define an analogue of the space of harmonic actions on the line for a group wich is not finitely generated.

\subsubsection*{Acknowledgements} We thank Adrien Le Boudec and Todor Tsankov for many useful conversations. Martín Gilabert Vio acknowledges support from the ECOS project 23003 “Small spaces under actions”. Nicol\'as Matte Bon is partially supported by ANR project PLAGE ANR-24-CE40-3137. 
\tableofcontents
\section{Preliminaries} \label{section:preliminaries} 
We review the theory of group actions on the line and some basic vocabulary on Borel equivalence relations. For more details on this material, see \cite{DeroinNavasRivas2016, Navas2011} and \cite{KB}. In this section $G$ will always denote a countable group.

\subsubsection*{Conventions}
A \emph{Radon measure on $\R$} is a non-zero regular Borel measure on $\R$, and thus gives finite mass to compact sets. We write $\mathrm{Homeo}_0(\R)$ for the group of orientation-preserving homeomorphisms of the real line. Given an action $\varphi \colon G \to \mathrm{Homeo}_0(\R)$ and a subgroup $H \subseteq G$ we write $\mathrm{Fix}_\varphi(H)$ for the closed set of $x \in \R$ such that $\varphi(h).x = x$ for all $h \in H$, and $\mathrm{supp}_\varphi(H) = \R \setminus \mathrm{Fix}_\varphi(H)$. The \emph{centralizer} of $\varphi(H)$ is the group $\mathrm{Cent}_\varphi(H) \subseteq \mathrm{Homeo}_0(\R)$ of homeomorphisms commuting with all elements in $\varphi(H)$.

If $K, H \subseteq G$ are subgroups, we say that $K$ and $H$ are \emph{commensurate} if $K\cap H$ has finite index in $K$ and in $H$, and we say that $H$ is \emph{commensurated in $G$} if for all $g \in G$, $H \cap gH g^{-1}$ has finite index in $H$ and in $gHg^{-1}$. For us, conjugation by some $a \in G$ is the map $g \in G \mapsto aga^{-1}$. We denote by $\langle \langle H \rangle \rangle$ the normal subgroup of $G$ generated by $H \subseteq G$.

\subsection{Actions on the line} 
An action $\varphi \in \mathrm{Hom}(G, \mathrm{Homeo}_0(\R))$ is said to be \emph{irreducible} if it has no global fixed points on $\R$, and we denote the set of irreducible actions of $G$ by $\mathrm{Rep}_\mathrm{irr}(G)$.
Two actions $\varphi_1, \varphi_2 \in \mathrm{Rep}_\mathrm{irr}(G)$ are \emph{semiconjugate}  if there is a non-decreasing map $h\colon \R\to \R$,  called a \emph{semiconjugacy}, such that $h \circ \varphi_1(g) = \varphi_2(g) \circ h$ for every $g \in G$. Semiconjugacy is an equivalence relation on $\mathrm{Rep}_\mathrm{irr}(G)$. Notice that the map $h$ is not necessarily continuous, however it is automatically a homeomorphism  when both actions are minimal. When the map $h$ is continuous, we shall refer to it as a \emph{continuous semiconjugacy from} $\varphi_1$ \emph{to} $\varphi_2$ (this defines a transitive but not symmetric relation).

We say that a closed non-empty subset $\Lambda \subseteq \R$ is a \emph{minimal set} for an action $\varphi \in \mathrm{Rep}_\mathrm{irr}(G)$ if $\Lambda$ is $\varphi$-invariant and every $\varphi$-orbit in $\Lambda$ is dense in $\Lambda$. The set $\Lambda$ may be either
\begin{itemize}
        \item a discrete orbit, in which case $\varphi$ is semiconjugate to an action that factors through $\Z$, or
        \item all of $\R$, or
        \item a perfect, totally disconnected and unbounded subset of $\R$, in which case by collapsing the intervals in $\R \setminus \Lambda$ we obtain a continuous semiconjugacy from $\varphi$ to a minimal action of $G$ on $\R$.
\end{itemize}
When $\Lambda$ is not a discrete orbit it is the unique minimal set, and if $G$ is finitely generated $\varphi$ always admits a minimal set (see \cite[Lemma 3.5.18]{DeroinNavasRivas2016}). In general, the existence of a minimal set for $\varphi$ is equivalent to $\varphi$ being \emph{cocompact}, that is, such that there is a compact subset of $\R$ intersecting every $\varphi$-orbit (see \cite[Proposition 2.1.2]{Navas2011}). In this case, any closed $\varphi$-invariant subset of $\R$ contains a minimal set.

Given $\varphi \in \mathrm{Rep}_\mathrm{irr}(G)$, two points $x,y \in \R$ are said to be \emph{proximal for $\varphi$} if there is $z \in \R$ and a sequence $(g_n)_{n \geq 0} \subseteq G$ such that $\lim_{n \to \infty} \varphi(g_n).x = \lim_{n \to \infty} \varphi(g_n).y = z$. The action $\varphi$ is said to be \emph{proximal} if every pair of points of $\R$ is proximal for $\varphi$, and is said to be \emph{locally proximal} if every point in $\R$ is contained in an open interval whose endpoints are proximal for $\varphi$.  When $\varphi$ admits a minimal set a further trichotomy is true.
\begin{teo}[\cite{Margulis2000, Antonov1984}] \label{teo:tits}
Let $\varphi \in \mathrm{Rep}_\mathrm{irr}(G)$ and suppose that $\varphi$ admits a minimal set. Then either:
\begin{enumerate}[label=(\Roman*)]
        \item the action is semiconjugate to an action by translations, or
           \item the action is locally proximal and not proximal, or
        \item the action is proximal.
     
\end{enumerate}
If $\varphi$ is proximal, then $G$ contains a non-abelian free semigroup. If $\varphi$ is locally proximal and not proximal, then $G$ contains a non-abelian free group.
\end{teo}

We say that $\varphi$ is of \emph{type I}, \emph{type II} or \emph{type III} according to which of the previous alternatives is verified. When $\varphi$ is minimal, these three cases are distinguished by their centralizer $\mathrm{Cent}_\varphi(G)$: the action $\varphi$ is of type I if $\mathrm{Cent}_\varphi(G)$ is conjugate to the group of translations, of type II if $\mathrm{Cent}_\varphi(G)$ is infinite cyclic, and of type III if $\mathrm{Cent}_\varphi(G)$ is trivial. Type I actions further decompose into those semiconjugate to a group of translations acting minimally on $\R$, and those that are semiconjugate to a \emph{cyclic} action, that is, an action factoring through a morphism $G \to \Z$.

\subsection{Borel equivalence relations} \label{subsection:introBorel}
A \emph{standard Borel space} is a measurable space isomorphic to the Borel $\sigma$-algebra of a \emph{Polish space}, that is, a complete separable metric space. An equivalence relation $E$ on a standard Borel space $Z$ is said to be \emph{Borel} if $E \subseteq Z \times Z$ is Borel, and is said to be \emph{countable} if all equivalence classes of $E$ are countable. We denote the $E$-equivalence class of an element $z \in Z$ by $[z]_E$. Recall that, given equivalence relations $E_1, E_2$ defined on standard Borel spaces $Z_1, Z_2$ respectively we say that \emph{$E_1$ is reducible to $E_2$} if there exists a Borel map $r \colon Z_1 \to Z_2$ such that for all $x,y \in Z_1$, we have $(x,y) \in E_1$ if and only if $(r(x), r(y)) \in E_2$, and that $E_1$ is \emph{bireducible} to $E_2$ if $E_1$ is reducible to $E_2$ and viceversa.

Studying the bireducibility class of the conjugacy relation for different classes of topological or measurable group actions is a question that has already spurred much research, see the monographs \cite{Hjorth2000, KB}. In our context we are only interested in the following two bireducibility types. As was stated in Section \ref{section:introduction}, a Borel equivalence relation $E$ on $Z$ is \emph{smooth} if there exists a Borel map $Z \to \R$ whose fibers are the $E$-classes. This condition is invariant under bireducibility, and it is equivalent to saying that the quotient measurable space $Z/E$ is standard. A larger class of Borel equivalence relations are the \emph{essentially hyperfinite} ones, which can be defined as the ones bireducible to a countable Borel equivalence relation $\widetilde{E}$ that is the orbit equivalence relation of a Borel action of the integers $\Z$ on a standard Borel space $Z$ \cite{SlamanSteel1988}.

A distinguished example of a non-smooth hyperfinite equivalence relation is $\cE_0$, the relation on $\{0,1\}^\N$ where $\left((x_n)_{n \geq 0}, (y_n)_{n \geq 0}\right) \in \cE_0$ if and only if there is an $m \in \N$ with $x_n = y_n$ for all $n \geq m$. A result of Dougherty-Jackson-Kechris \cite{DoughertyJacksonKechris1994} states that $\cE_0$ is the unique non-smooth and essentially hyperfinite Borel equivalence relation up to bireducibility. Moreover, $\cE_0$ is reducible to \emph{any} non-smooth (and not necessarily essentially hyperfinite) Borel equivalence relation \cite{HarringtonKechrisLouveau1990}. We will make use of this theorem, called the \emph{Glimm-Effros dichotomy}, in the special case when the Borel equivalence relation is $F_\sigma$ (that is, a countable union of closed sets) and is generated by the action of a Polish group on a Polish space.

\begin{teo}[{\cite{Effros1981}, see \cite[Theorem 3.4.2]{KB}}] \label{teo:glimmEffrosIntro}
Let $\Gamma$ be a Polish group and $Y$ a Polish space equipped with a continuous action of $\Gamma$. Let $E \subseteq Y \times Y$ be the orbit equivalence relation of the action, and suppose that $E$ is $F_\sigma$. Then either $E$ is smooth or $\cE_0$ is reducible to  $E$.

Moreover, the first alternative holds if and only if for every $y \in Y$, the map
\[
	\gamma \mathrm{Stab}_\Gamma(y) \in \Gamma/\mathrm{Stab}_\Gamma(y) \mapsto \gamma.y \in \mathrm{Orb}_\Gamma(y)
\]
is a homeomorphism.
\end{teo}

\begin{ejm} \label{ejm:F2}
        We restate an example from \cite[Section 14.4]{BMRT} showing that $\cE_0$ Borel reduces to conjugacy among minimal (even faithful) actions of the free group $F_2$ on the line. Consider $\widetilde{g},\widetilde{h}$ two homeomorphisms of $\R/\Z$ with $\mathrm{Fix}(\widetilde{g}) = \{0, 1/2\}$, $\mathrm{Fix}(\widetilde{h}) = \{1/4, 3/4\}$ such that $\langle \widetilde{g}, \widetilde{h} \rangle$ acts minimally on $\R/\Z$. Let $g,h$ be lifts of $\widetilde{g}, \widetilde{h}$ to the line with fixed points $\mathrm{Fix}(g) = \frac{1}{2}\Z$ and $\mathrm{Fix}(h) = 1/4 + \frac{1}{2}\Z$. Given a word $\omega = (\omega_n)_{n \in \Z} \in \{\pm 1\}^\Z$ define a homeomorphism $g_\omega$ by $g_\omega(x) = g^{\omega_n}(x)$ for every $x \in [n/2, (n+1)/2], \, n \in \Z$, and let $\varphi_\omega$ be the action of $F_2$ defined by $g_\omega, h$. Then the orbits of $\varphi_\omega$ are the same of those of $\langle g, h \rangle$, so $\varphi_\omega$ is minimal. Two such actions $\varphi_\omega,\, \varphi_{\omega'}$ are conjugate if and only if $\omega, \omega'$ belong to the same orbit of the bilateral shift: indeed, any conjugacy between $\varphi_\omega, \, \varphi_{\omega'}$ preserves $\frac{1}{2}\Z$, and for every $n \in \Z$ the condition $\omega_n = 1$ can be recognized from the sign of $g_\omega - \mathrm{id}_\R$ on $[n/2, (n+1)/2]$.
\end{ejm}

We will also need the following strong uniformization theorem. Recall that a $K_\sigma$ set in a Polish space is a countable union of compact sets.

\begin{teo}[{\cite{Arsenin1940, Kunugui1940}, see \cite[(18.18)]{Kechris1995}}] \label{teo:ksigma}
Let $Z$ be a standard Borel space, $Y$ a Polish space and $B \subseteq Y \times Z$ a Borel subset. Denote by $\pi \colon Y \times Z \to Y$ the projection onto the first coordinate. Suppose that the sections $\pi^{-1}(y) \cap B$ are $K_\sigma$ for every $y \in Y$. Then $\pi(B)$ is a Borel subset, and there exists a Borel function $\zeta \colon \pi(B) \to B$ such that $\pi \circ \zeta = \mathrm{id}_{\pi(B)}$.
\end{teo}

\subsection{The space of harmonic actions on the line} \label{sec:harmonic}
Let $G$ be a finitely generated group. Equip $\mathrm{Rep}_\mathrm{irr}(G)$ with the topology coming from the inclusion $\mathrm{Rep}_\mathrm{irr}(G) \subseteq \mathrm{Homeo}_0(\R)^G$ and define the \emph{translation flow} $\R \curvearrowright^\Psi \mathrm{Rep}_\mathrm{irr}(G,\R)$ by
\[
        \Psi^t(\varphi)(g) = T_{t} \circ \varphi(g)\circ T_{-t}
\]
for all $g \in G$ and $\varphi \in \mathrm{Rep}_\mathrm{irr}(G)$, where $T_t\colon s\mapsto s+t$ is the translation. Notice that $\Psi$ defines a continuous action of $\R$.

We say that two actions $\varphi_1, \varphi_2 \in \mathrm{Rep}_\mathrm{irr}(G)$ are \emph{pointed semiconjugate} if there is an action $\eta \in \mathrm{Rep}_\mathrm{irr}(G)$ that is minimal or cyclic, and semiconjugacies $h_i \colon \R \to \R$ between $\varphi_i$ and $\eta$ for $i = 1,2$ such that $h_1(0) = h_2(0)$.

The following result is a consequence of the fundamental work of Deroin, Kleptsyn, Navas and Parwani \cite{DKNP}.
\begin{teo}[{\cite[Theorem 8.5]{DKNP}, \cite[Section 14.2]{BMRT}}] \label{teo:existenciaDeroin}
        Let $G$ be a finitely generated group. Then there exists a compact $\Psi$-invariant subspace $\mathrm{Harm}(G) \subseteq \mathrm{Rep}_\mathrm{irr}(G)$ composed of minimal and cyclic actions, and a continuous retraction $r \colon \mathrm{Rep}_\mathrm{irr}(G) \to \mathrm{Harm}(G)$ such that two actions $\varphi_1, \varphi_2 \in \mathrm{Rep}_\mathrm{irr}(G)$ are pointed semiconjugate if and only if $r(\varphi_1) = r(\varphi_2)$. In particular, $\varphi_1, \varphi_2$ are semiconjugate if and only if $r(\varphi_1), r(\varphi_2)$ are in the same $\Psi$-orbit.
\end{teo}

The statement follows entirely from \cite[Theorem 8.5]{DKNP}, except for the continuity of the map $r$, which is proven in \cite[Section 14.2]{BMRT}. A subspace $\mathrm{Harm}(G)$ satisfying the conclusion of the theorem is not unique: in fact the construction in \cite{DKNP} depends on the choice of symmetric finitely supported probability measure $\mu$ on $G$ (indeed  $\mathrm{Harm}(G)$ is then realised as the space of normalised $\mu$-harmonic action). However the continuity of the map $r$ implies that any two subspaces satisfying the conclusion of Theorem \ref{teo:existenciaDeroin} are homeomorphic through a homeomorphism sending $\Psi$-orbits to $\Psi$-orbits (i.e. a \emph{flow equivalence}). In other words, the space with flow $(\mathrm{Harm}(G), \Psi)$ is uniquely defined up to flow equivalence \cite[Theorem 3.7]{BMRTrealisation}. With this in mind, we refer to the space $(\mathrm{Harm}(G), \Psi)$ as \emph{the space of harmonic actions}, or the \emph{harmonic space}, of $G$.

We also recall  that the classification of actions into type I, II, III from Theorem \ref{teo:tits} reads in the space $(\mathrm{Harm}(G), \Psi)$ as follows.

\begin{prop}[{see \cite[Proposition 3.10]{BMRTrealisation}}]
Let $G$ be a finitely generated group, and $\varphi\in\mathrm{Rep}_\mathrm{irr}(G)$. Then:
\begin{itemize}
    \item $\varphi$ is of type I if and only if $r(\varphi)
    \in \mathrm{Harm}(G)$ is a fixed point of the flow $\Psi$. The set of fixed points of $\Psi$ consists of actions by translations, and is homeomorphic to the character sphere $\mathbb{S}^{b_1(G)-1}$ , with  $b_1(G)=\dim(\mathrm{Hom}(G, \R))$   (and $\mathbb{S}^{-1} = \varnothing$).
    
    \item $\varphi$ is of type II if and only if $r(\varphi)\in \mathrm{Harm}(G)$ is a periodic non-fixed point of $\Psi$. 
    
    \item $\varphi$ is of type III if and only if $r(\varphi)\in \mathrm{Harm}(G)$ is a non-periodic point of $\Psi$.

\end{itemize}

\end{prop}

\begin{rmk} \label{rem:bounded-displacement}
Notice that for any $g \in G$ and $\varphi \in \mathrm{Harm}(G)$ the constant
\[
    \sup\{\abs{\varphi(g).x - x} : x \in \R\}
\]
is finite. Indeed, for every $g \in G$, by $\Psi$-invariance of $\mathrm{Harm}(G)$ we have
 \begin{align*}
 	\sup\{\abs{\varphi(g).x - x} : \varphi \in \mathrm{Harm}(G), \, x \in \R\} &= \sup\{\abs{\Psi^t(\varphi)(g).0} : \varphi \in \mathrm{Harm}(G),\, t \in \R\} \\
 		& = \sup\{\abs{\varphi(g).0} : \varphi \in \mathrm{Harm}(G)\},
\end{align*}
which is finite by compactness.
\end{rmk}

\subsection{The space of harmonic actions and the class $\cC$} \label{subsection:deroinborel}
Theorem \ref{teo:existenciaDeroin} implies that when the group $G$ is finitely generated, there is a Borel reduction of semiconjugacy on $\mathrm{Rep}_\mathrm{irr}(G)$ (in particular of conjugacy on $\mathrm{Rep}_\mathrm{min}(G)$ to the orbit equivalence relation of flow on a compact space $(\mathrm{Harm}(G), \Psi)$. Such a relation is essentially hyperfinite \cite{Wagh1988}, and as a consequence the previous theorem exhibits a dichotomy for finitely generated groups acting on the line: either their semiconjugacy relation is smooth or it is bireducible to $\cE_0$. In this context, Theorem \ref{teo:glimmEffrosIntro} already gives a criterion for finitely generated groups to belong to $\cC$ since the orbit equivalence relations of an $\R$-flow is always $F_\sigma$. Here, a \emph{recurrent point} of $\Psi$ is an action $\varphi \in \mathrm{Harm}(G)$ such that for every open $U \subseteq \mathrm{Harm}(G)$ the set of return times $\{t \in \R : \Psi^t(\varphi) \in U\}$ is unbounded.

\begin{prop}[{\cite[Corollary 14.4.2]{BMRT}}] \label{prop:recurrenteIntro}
        A finitely generated group $G$ belongs to $\cC$ if and only if every $\Psi$-recurrent element of $\mathrm{Harm}(G)$ is $\Psi$-periodic. If this is not the case, then the semiconjugacy relation on $\mathrm{Rep}_\mathrm{irr}(G)$ (equivalently, the conjugacy relation on the space of minimal actions of $G$) is bi-reducible to $\mathcal{E}_0$.
\end{prop}
The finite generation assumption on $G$ is crucial in this discussion. We do not know whether the last assertion in the previous proposition holds for general countable groups. In Subsection \ref{subsection:deroinFinf}, we show that Theorem \ref{teo:existenciaDeroin} is simply not true for countable groups.
 In Section \ref{section:complexity} we show that the semiconjugacy relation on $\mathrm{Rep}_\mathrm{irr}(G)$ is always essentially countable.

Proposition \ref{prop:recurrenteIntro} also makes apparent the connection between Question \ref{qn:amenable} and Witte Morris' theorem stating that every finitely generated amenable group $G$ that acts non-trivially on the line has a non-trivial homomorphism to $\R$ \cite{WitteMorris2006}. In fact, elaborating on an idea of B. Deroin \cite{deroinAlmost}, it was shown in \cite[Remark 3.11]{BMRTrealisation} that  this theorem is essentially equivalent to the following statement. Recall a that \emph{uniformly recurrent point} of $\Psi$ is a point $\varphi \in \mathrm{Harm}(G)$ such that for every open $U \subseteq \mathrm{Harm}(G)$, there exists $C>0$ such that the set of return times $\{t \in \R : \Psi^t(\varphi) \in U\}$ intersects every interval $I\subset \R$ of length $|I|\ge C$. This is equivalent to the condition that $\varphi$ belongs to a closed minimal $\Psi$-invariant subset: in particular uniformly recurrent points always exist. 
\begin{teo}[Witte-Morris' theorem revisited]\label{teo:wittemorris}
Let $G$ be a finitely generated amenable group, and $(\mathrm{Harm}(G), \Psi)$ its space of harmonic actions on the line. Then every uniformly recurrent point $\varphi\in\mathrm{Harm}(G)$ of the flow $\Psi$ is a fixed point, i.e. an action by translations.
\end{teo}
\begin{proof}
The group $G$ has a natural action on $\mathrm{Harm}(G)$ preserving each $\Psi$ orbit, hence minimal closed $\Psi$-invariant subsets are also $G$-invariant. If $G$ is amenable, every closed $G$-invariant subset supports a $G$-invariant probability measure. But it is shown in \cite[Proposition 3.10]{BMRTrealisation} that all $G$-invariant probability measures are supported on fixed points. Hence the only closed minimal $\Psi$-invariant subsets are fixed points.
\end{proof}
 From Proposition \ref{prop:recurrenteIntro}, it is then apparent that Question \ref{qn:amenable}~\eqref{i-amenable} amounts to ask whether Theorem~\ref{teo:wittemorris} remains true when the word ``uniformly'' is dropped.

\section{Criteria for membership in $\cC$} \label{section:spaces}
This section is devoted to the proof of Theorem \ref{cor:criterio}, which is the main criterion used in subsequent sections to decide membership in $\cC$. Subsection \ref{subsection:spaces} introduces the spaces of irreducible actions of interest, Subsection \ref{subsection:approxConj} proves preliminary lemmas enabling the use of the Glimm-Effros dichotomy, and Subsection \ref{subsection:almostCent} proves Theorem \ref{cor:criterio}. In this section, $G$ is always a countable group.

\subsection{Spaces of actions} \label{subsection:spaces} The group $\mathrm{Homeo}_0(\R)$ is naturally identified with $\mathrm{Homeo}_0([0,1])$, and a result of R. Arens \cite[Theorems 1 \& 5]{Arens1946} implies that this identification is a homeomorphism when both groups are given the compact-open topology. Hence $\mathrm{Homeo}_0(\R)$ with the compact-open topology is a Polish group, and the inequality
\[
        \sup_{z \in [x,y]} \abs{f(z) - g(z)} \leq \max(f(y), g(y)) - \min(f(x), g(x)),
\]
valid for $x\leq y$ and $f,g\in \mathrm{Homeo}_0(\R)$, shows that this topology coincides with the topology of pointwise convergence. 

Fix $G$ a countable group and endow $\mathrm{Hom}(G, \mathrm{Homeo}_0(\R))$ with the induced topology from the product topology on $\mathrm{Homeo}_0(\R)^G$. Define:
\begin{itemize}
        \item $\mathrm{Rep}_\mathrm{cc}(G) \subseteq \mathrm{Rep}_\mathrm{irr}(G)$ the space of cocompact actions of $G$,
        \item $\mathrm{Rep}_\mathrm{min}(G), \mathrm{Rep}_\mathrm{cyc}(G) \subseteq \mathrm{Rep}_\mathrm{cc}(G)$ the space of minimal actions of $G$ and the space of cyclic actions of $G$, and
        \item $\mathrm{Rep}_\mathrm{III}(G), \mathrm{Rep}_\mathrm{cent}(G) \subseteq \mathrm{Rep}_\mathrm{min}(G)$ the space of minimal type III actions of $G$ and the space of minimal actions of $G$ with non-trivial centralizer.
\end{itemize}
The inclusions between these spaces are summarized in Figure \ref{fig:inclusions}.

\begin{figure}[h] 
\centering
\begin{tikzpicture}[scale=1]
\draw(0,0) rectangle (9.5,6)node[below left]{$\mathrm{Rep}_\mathrm{irr}(G)$};
\draw (0.5,0.5) rectangle (9,5) node[below left] {$\mathrm{Rep}_\mathrm{cc}(G)$};
\draw (1,1) rectangle (3,3);
\draw (2,2) node {$\mathrm{Rep}_\mathrm{III}(G)$};
\draw (3,1) rectangle (5,3);
\draw (4,2) node {$\mathrm{Rep}_\mathrm{cent}(G)$};
\draw (6,1) rectangle (8,3);
\draw (7,2) node {$\mathrm{Rep}_\mathrm{cyc}(G)$};
\draw (3,3) node[above] {$\mathrm{Rep}_\mathrm{min}(G)$};
\end{tikzpicture}
\caption{}
\label{fig:inclusions}
\end{figure}

By writing
\[
	\mathrm{Rep}_\mathrm{min}(G) = \bigcap_{\substack{q \in \Q, n \in \N_+, \\ m \in \Z}} \bigcup_{\substack{F \subseteq G\\ \text{ finite}} } \{\varphi \colon G \to \mathrm{Homeo}_0(\R) : \varphi(F).(q-1/n, q + 1/n) \supseteq [m,m+1]\}
\]
and
\[
	\mathrm{Rep}_\mathrm{III}(G) = \bigcap_{\substack{p,q,r,s \in \Q \\ p<q,\, r< s}} \,  \bigcup_{g \in G} \{\varphi \colon G \to \mathrm{Homeo}_0(\R) : \varphi(g).[p,q] \subseteq (r,s)\}
\]
we see that $\mathrm{Rep}_\mathrm{min}(G)$ and $\mathrm{Rep}_\mathrm{III}(G)$ are $G_\delta$ sets inside $\mathrm{Hom}(G, \mathrm{Homeo}_0(\R))$, so they are also Polish for the induced topology \cite[(3.11)]{Kechris1995}.

We consider the continuous action of $\mathrm{Homeo}_0(\R)$ on $\mathrm{Rep}_\mathrm{min}(G)$ by conjugation, defined by $(f.\varphi)(g) = f \circ \varphi(g) \circ f^{-1}$ for every $f \in \mathrm{Homeo}_0(\R)$, $\varphi \in \mathrm{Rep}_\mathrm{min}(G)$ and $g \in G$. The stabilizer $\mathrm{Stab}_{\mathrm{Homeo}_0(\R)}(\varphi)$ of an action $\varphi \in \mathrm{Rep}_\mathrm{min}(G)$ coincides with its centralizer $\mathrm{Cent}_\varphi(G)$, which is thus closed in $\mathrm{Homeo}_0(\R)$.

\subsection{Approximate conjugacies} \label{subsection:approxConj}

\begin{lema} \label{lema:preorden}
Consider minimal actions $\varphi,\psi \in \mathrm{Rep}_\mathrm{min}(G)$ such that there are sequences $(f_n)_{n \geq 0} \subseteq \mathrm{Homeo}_0(\R)$ and $(\varphi_n)_{n\geq 0} \subseteq \mathrm{Rep}_\mathrm{min}(G)$ with $\lim_{n \to \infty} \varphi_n = \varphi$ and $\lim_{n \to \infty} f_n. \varphi_n = \psi$. If there exists an $x \in \R$ such that $(f_n^{-1}(x))_{n \geq 0}$ admits a limit point $y \in \R$, then there exists $f \in \mathrm{Homeo}_0(\R)$ with $f.\varphi = \psi$ and $f(y) = x$.
\end{lema}
\begin{proof}
We claim that for every $g \in G$, the inequality $\varphi(g).y > y$ implies $\psi(g).x \geq x$: indeed, if $\varphi(g).y > y$ then for sufficiently large $n \in \N$ we have $\varphi_n(g).y > y$ and $\varphi_n(g).\left( f_n^{-1}(x) \right) > f_n^{-1}(x)$, so $(f_n. \varphi_n)(g).x > x$ and hence $\psi(g).x \geq x$.  It also follows that $\varphi(g).y < y$ implies that $\psi(g).x \leq x$, and in particular $\varphi(g).y=y$ implies that $\psi(g).x = x$. Hence the correspondence $\varphi(g).y\mapsto \psi(g).x$ is a well-defined, non-decreasing equivariant map from $\varphi(G).x\to  \psi(G).y$. By \cite[Lemma 2.3]{BrumMatteBonRivasTriestino2025}, this map extends to a homeomorphism $f\colon \R\to \R$ conjugating $\varphi$ to $\psi$. \qedhere
\end{proof}

\begin{prop}\label{prop:fsigma}
The orbit equivalence relation 
\[
	E^\mathrm{min}_G = \{ (\varphi, \psi) \in \mathrm{Rep}_\mathrm{min}(G) \times \mathrm{Rep}_\mathrm{min}(G) : \text{there exists }f \in \mathrm{Homeo}_0(\R) \text{ such that }f. \varphi = \psi\}
\]
is $F_\sigma$ in $\mathrm{Rep}_\mathrm{min}(G) \times \mathrm{Rep}_\mathrm{min}(G)$.
\end{prop}
\begin{proof}
	For $k \in \N_+$ define $F_k$ as the set of $(\varphi, \psi) \in E^\mathrm{min}_G$ such that there exists $f \in \mathrm{Homeo}_0(\R)$ with $\abs{f^{-1}(0)} \leq k$ and $f.\varphi = \psi$. We will prove that each $F_k$ is closed in $\mathrm{Rep}_\mathrm{min}(G) \times \mathrm{Rep}_\mathrm{min}(G)$, implying the desired conclusion since $E^\mathrm{min}_G = \bigcup_{k \geq 0} F_k$. 

Let $((\varphi_n, f_n. \varphi_n))_{n \geq 0} \subseteq F_k$ be a sequence such that $\lim_{n \to \infty} \varphi_n  = \varphi$ and $\lim_{n \to \infty} f_n. \varphi_n = \psi$ for some $(\varphi, \psi) \in E^\mathrm{min}_G$. Up to considering a subsequence, we may assume that there exists $y \in [-k, k]$ such that $\lim_{n \to \infty} f_n^{-1}(0) = y$. By Lemma \ref{lema:preorden} we produce $f \in \mathrm{Homeo}_0(\R)$ such that $f(y) = 0$ and $f.\varphi = \psi$, so $(\varphi, \psi) \in F_k$.
\end{proof}

\begin{rmk}
Proposition \ref{prop:fsigma} implies that the conjugacy relation on $\mathrm{Rep}_\mathrm{min}(G)$ and on all its subspaces we have considered are Borel. This conclusion can be obtained directly through an argument close to \cite[Lemma 6.5]{JahelJoseph2025}: by \cite[Theorem 7.1.2]{KB} this is equivalent to seeing that the map
\[
	C \colon \varphi \in \mathrm{Rep}_\mathrm{min}(G) \mapsto \mathrm{Stab}_{\mathrm{Homeo}_0(\R)}(\varphi) = \mathrm{Cent}_\varphi(G) \in \mathrm{Sub}(\mathrm{Homeo}_0(\R))
\]
is Borel, where $\mathrm{Sub}(\mathrm{Homeo}_0(\R))$ is the standard Borel space of closed subgroups of $\mathrm{Homeo}_0(\R)$, endowed with the Effros Borel structure generated by the sets
\[
	B_U = \{H \in \mathrm{Sub}(\mathrm{Homeo}_0(\R)) : U \cap H \neq \emptyset \}
\]
for open $U \subseteq \mathrm{Homeo}_0(\R)$ (see \cite[Proposition 1]{Malicki2011}).

Write
\[	
	C^{-1}(B_U) = \{\varphi \in \mathrm{Rep}_\mathrm{min}(G) : U \cap \mathrm{Cent}_\varphi(G) \neq \emptyset\} = \pi(B)
\]
where $B = \{ (\varphi, g) \in \mathrm{Rep}_\mathrm{min}(G) \times U : g \in \mathrm{Cent}_\varphi(G) \}$ and $\pi \colon \mathrm{Rep}_\mathrm{min}(G) \times U \to \mathrm{Rep}_\mathrm{min}(G)$ is the projection onto the first coordinate. The space $\mathrm{Rep}_\mathrm{min}(G) \times U$ is Polish and $B$ is Borel, and moreover for every $\varphi \in \mathrm{Rep}_\mathrm{min}(G)$ the section $B \cap \pi^{-1}(\varphi)$ is homeomorphic to $\mathrm{Cent}_\varphi(G) \cap U$. Since $\mathrm{Cent}_\varphi(G)$ is homeomorphic either to a point, $\Z$ or $\R$, we have that $B \cap \pi^{-1}(\varphi)$ is always $K_\sigma$, so Theorem \ref{teo:ksigma} implies that $\pi(B)$ is Borel. We conclude that $C$ is a Borel map.
\end{rmk}

\subsection{Almost centralizing sequences and groups in $\cC$} \label{subsection:almostCent}

\begin{defn}
Let $\varphi\in \mathrm{Rep}_\mathrm{irr}(G)$. We say that a sequence  $(f_n)_{n \geq 0} \subseteq \mathrm{Homeo}_0(\R)$ \emph{almost centralizes} $\varphi$ if $\lim_{n \to \infty} f_n.\varphi = \varphi$.
\end{defn}
The following will be our main tool to decide membership to $\mathcal{C}$.
\begin{teo} \label{cor:criterio}

For a countable group $G$, the following are equivalent:
\begin{enumerate}
    \item \label{i-GinC} $G\in \cC$;
    \item \label{i-approxtrivial} for every $\varphi \in \mathrm{Rep}_\mathrm{III}(G)$, every sequence $(f_n)_{n \geq 0}$ almost centralizing $\varphi$ tends to $\mathrm{id}_\R$;
    \item \label{i-approxbounded} for every $\varphi \in \mathrm{Rep}_\mathrm{III}(G)$ and every sequence $(f_n)_{n \geq 0}$ almost centralizing $\varphi$, there exists $z \in \R$ such that $(f_n^{-1}(z))_{n \geq 0}$ is bounded.
\end{enumerate}

If $G$ is finitely generated, then these conditions are also equivalent to:
\begin{enumerate}[resume]
    \item\label{i-approxderoin} for every $\varphi \in \mathrm{Rep}_\mathrm{III}(G) \cap \mathrm{Harm}(G)$ and $(t_n)_{n \geq 0} \subseteq \R$ such that $\lim_{n \to \infty} \Psi^{t_n}(\varphi) = \varphi$ we have $\lim_{n \to \infty} t_n = 0$.
\end{enumerate}
\end{teo}
\begin{proof}
The equivalence between \eqref{i-GinC} and \eqref{i-approxderoin} follows from Proposition \ref{prop:recurrenteIntro}. We now concentrate on the other equivalences.

Notice that the conjugacy relation on $\mathrm{Rep}_\mathrm{min}(G)$ is precisely the orbit equivalence relation $E^\mathrm{min}_G$ of $\mathrm{Homeo}_0(\R)$ acting on $\mathrm{Rep}_\mathrm{min}(G)$ by conjugation. Since $E^\mathrm{min}_G$ is $F_\sigma$ by Proposition \ref{prop:fsigma}, the Glimm-Effros dichotomy shows that $E^\mathrm{min}_G$ is smooth if and only if for every $\varphi \in \mathrm{Rep}_\mathrm{min}(G)$, the orbital map
\[
	\Psi_\varphi \colon f \in \mathrm{Homeo}_0(\R)/\mathrm{Cent}_\varphi(G) \mapsto f.\varphi \in \mathrm{Orb}_{\mathrm{Homeo_0(\R)}} (\varphi)
\]
is a homeomorphism (it is always continuous). Thus if $G$ is in $\cC$, then sequence $(f_n)_{n \geq 0}$  almost centralizing some $\varphi \in \mathrm{Rep}_\mathrm{III}(G)$ must verify $\lim_{n \to \infty} f_n = \mathrm{id}_\R$, and in particular all the sequences $(f_n^{-1}(z))_{n \geq 0},\, z \in \R$ are bounded. This proves that \eqref{i-GinC}$\Rightarrow$\eqref{i-approxtrivial}$\Rightarrow$\eqref{i-approxbounded}.

To prove the remaining implication \eqref{i-approxbounded}$\Rightarrow$\eqref{i-GinC}, suppose that every action in $\mathrm{Rep}_\mathrm{III}(G)$ satisfies the condition in the statement of the proposition, and let $\varphi \in \mathrm{Rep}_\mathrm{min}(G)$. We will show that $\Psi_\varphi^{-1}$ is continuous, and this suffices to prove that $G \in \cC$ by the previous paragraph.

Suppose first that $\varphi \in \mathrm{Rep}_\mathrm{III}(G)$, and let $(f_n)_{n \geq 0} \subseteq \mathrm{Homeo}_0(\R)$ with $\lim_{n \to \infty} f_n.\varphi = \varphi$. Then there is a $z \in \R$ such that $(f_n^{-1}(z))_{n \geq 0}$ is bounded by the hypothesis.  Fix $x \in \R$ and let $g \in G$ such that $\varphi(g).x < z$. Then for large enough $n \in \N$ we have $f_n \circ \varphi(g) \circ f_n^{-1}(x) < z$, or $f_n^{-1}(x) < \varphi(g^{-1}) \circ f_n^{-1}(z)$. Hence $(f_n^{-1}(x))_{n \geq 0}$ is bounded above, and a symmetrical argument shows that it is also bounded below. 

Let $y \in \R$ be a limit point of the sequence $(f_n^{-1}(x))_{n \geq 0}$. Lemma \ref{lema:preorden} shows that there is a $f \in \mathrm{Homeo}_0(\R)$ such that $f(x) = y$ and $f.\varphi = \varphi$. Since $\varphi$ is of type III, $f$ must be $\mathrm{id}_\R$ and $y = x$. Thus $\lim_{n \to \infty} f_n^{-1}(x) = x$, and as this is true for every $x \in \R$, we have $\lim_{n \to \infty} f_n = \mathrm{id}_\R$. Thus $\Psi_\varphi^{-1}$ is continuous in this case.

Suppose that $\varphi \in \mathrm{Rep}_\mathrm{cent}(G)$ instead, and let $(f_n)_{n \geq 0} \subseteq \mathrm{Homeo}_0(\R)$ with $\lim_{n \to \infty} f_n.\varphi = \varphi$. If $\varphi$ is of type I, then we can find $(c_n)_{n \geq 0} \subseteq \mathrm{Cent}_\varphi(G)$ such that $\lim_{n \to \infty} c_n \circ f_n^{-1}(0) = 0$. Suppose instead that $\varphi$ is of type II. Let $c$ be the generator of $\mathrm{Cent}_\varphi(G)$ with $c(0) > 0$, and take $(c_n)_{n \geq 0} \subseteq \mathrm{Cent}_\varphi(G)$ such that $c_n \circ f_n^{-1}(0) \in [c^{-1}(c(0)/2), c(0)/2)$ for every $n \in \N$. In any case, we have
\[
	\lim_{n \to \infty}(f_n \circ c_n^{-1}).\varphi = \lim_{n \to \infty} f_n.\varphi = \varphi,
\]
and Lemma \ref{lema:preorden} shows that any limit point $y$ of $c_n \circ f_n^{-1}(0)$ is of the form $f(0)$ for some function $f \in \mathrm{Cent}_\varphi(G)$. Hence $y = 0$, and since this is true for any limit point we deduce that
\[
	\lim_{n \to \infty} c_n \circ f_n^{-1}(0) = 0
\]
also in this case.

Now assume towards a contradiction that for some $x > 0$ we have $\limsup_{n \to \infty} c_n \circ f_n^{-1}(x) > x$, and upon passing to a subsequence we may assume that $\liminf_{n \to \infty} c_n \circ f_n^{-1}(x) > z$ for some $z > x$. Take $0 < \epsilon < x$ small enough so that there is $g \in G$ with $\varphi(g).[-\epsilon, \epsilon] \subseteq (x,z)$. But for $n$ large enough we have that $c_n \circ f_n^{-1}(0) \in [-\epsilon, \epsilon]$, so $\varphi(g) \circ c_n \circ f_n^{-1}(0) \in (x,z)$. For $n$ large enough we conclude that
\[
	(f_n.\varphi)(g).0 = f_n \circ c_n^{-1} \circ \varphi(g) \circ c_n \circ f_n^{-1}(0) < x,
\]
contradicting the assumption $\lim_{n \to \infty}f_n.\varphi(g) = \varphi(g)$. Hence $\limsup_{n \to \infty} c_n \circ f_n^{-1}(x) \leq x$ for every $x > 0$. 

Similarly, assume that for some $x > 0$ we have $\liminf_{n \to \infty} c_n \circ f_n^{-1}(x) < x$, so upon passing to a subsequence we have $\limsup_{n \to \infty} c_n \circ f_n^{-1}(x) < z$ for some $z \in (0, x)$. Take $0 < \epsilon < z$ small enough so that there is $g \in G$ with $\varphi(g).[-\epsilon, \epsilon] \subseteq (z,x)$. Again, for $n$ large enough we have $c_n \circ f_n^{-1}(0) \in [-\epsilon, \epsilon]$ and $\varphi(g) \circ c_n \circ f_n^{-1}(0) \subseteq (x,z)$. We conclude that for $n$ large the inequality 
\[
	(f_n.\varphi)(g).0 = f_n \circ c_n^{-1} \circ \varphi(g) \circ c_n \circ f_n^{-1}(0) > x
\]
holds, contradicting $\lim_{n \to \infty}f_n.\varphi(g) = \varphi(g)$. Hence $\lim_{n \to \infty}c_n \circ f_n^{-1}(x) = x$.

A symmetrical argument gives $\lim_{n \to \infty}c_n \circ f_n^{-1}(x) = x$ for all $x < 0$ also. We deduce that $\lim_{n \to \infty}c_n \circ f_n^{-1} = \mathrm{id}_\R$ in $\mathrm{Homeo}_0(\R)$ and that $\lim_{n \to \infty} f_n\mathrm{Cent}_\varphi(G) = \mathrm{Cent}_\varphi(G)$. We conclude that $\Psi_\varphi^{-1}$ is continuous.
\end{proof}

The previous theorem and Theorem \ref{teo:tits} imply the following, which was known for finitely generated groups as a consequence of Proposition \ref{prop:recurrenteIntro}.
\begin{cor} \label{cor:abeliano}
If $G$ does not admit any minimal proximal action on the line (in particular, if $G$ has no non-abelian free semigroups), then $G \in \cC$.
\end{cor}


\section{Actions on the line and commensurated subgroups} \label{section:actionsCommensurated}
In this section we collect several structural statements for minimal actions on the line of a countable group admitting a commensurated subgroup, to be used in the next section. These are natural generalisations of the corresponding statements for normal subgroups. Subsection \ref{subsection:conrad} describes the case when the action of the commensurated subgroup is of type I, and Subsection \ref{subsection:laminar} describes the case when it admits no minimal set. In this section $G$ is always a countable group.

\subsection{Commensurated subgroups and isolator} \label{subsection:isolator}

Recall that two subgroups $H, K$ of a group $G$ are \emph{commensurate} if $H\cap K$ has finite index in $G$. The subgroup $H$ is \emph{isolated} if whenever $g^n\in H$ for some $g\in G,\  n\in \Z$, then $g\in H$. The smallest isolated subgroup containing $H$ is called the \emph{isolator} of $H$, and is denoted by $\mathrm{I}(H)$. 
The following straightforward lemma explains why this notion arises here.

\begin{lema} \label{lema:isolator}
    Let $H$ be a commensurated subgroup of $G$. Then its isolator $\mathrm{I}(H)$ does not depend on the choice of $H$ in its commensurability class, and it is equal to the smallest normal subgroup of $G$ containing $H$ and such that $G/\mathrm{I}(H)$ is torsion free.
    
    In particular, every $\varphi\colon G
    \to \mathrm{Homeo_0(\R)}$ such that $\restr{\varphi}{H}$ is trivial factors through $G/\mathrm{I}(H)$.

\end{lema}

\subsection{Conrad homomorphisms and affine actions} \label{subsection:conrad}
We first describe actions semiconjugate to actions by translations, following \cite[Section 2.2]{Navas2011}. An action $\varphi \in \mathrm{Rep}_\mathrm{irr}(G)$ is of type I if and only if it preserves a non-trivial Radon measure $\nu$ on $\R$. In this, case the map $\tau_\nu \colon G \to \R$ given by
            \begin{equation} \label{eq:defTau}
                g \in G \mapsto \tau_\nu(g) = 
                \begin{cases}
                        \nu[x,\varphi(g).x) & \text{ if } x \leq \varphi(g).x\\
                        -\nu[\varphi(g).x, x) & \text{ if } x > \varphi(g).x
                \end{cases}
        \end{equation}
is well defined and does not depend on $x \in \R$. It is actually a group morphism called the \emph{Conrad homomorphism associated to $\varphi$}, which is unique up to positive rescaling.

\begin{prop}[see {\cite[Section 2.2]{Navas2011}}]\label{prop:conrad} 
        Let $\varphi \in \mathrm{Rep}_\mathrm{irr}(G)$ be an action of type I with a minimal set $\Lambda$, and let $\nu$ be a non-trivial Radon measure on $\R$ invariant under $\varphi(G)$.
\begin{enumerate}
        \item \label{itC:3} The action $\varphi$ is semiconjugate to the action by translations defined by $\tau_\nu$, and $\mathrm{ker}(\tau_\nu)$ coincides with the elements $h \in H$ such that $\varphi(h)$ is trivial on $\Lambda$. In particular, $\tau$ is trivial if and only if $\varphi$ has a global fixed point.
        \item \label{itC:4} If $\nu'$ is another non-zero Radon measure preserved by $\varphi(G)$ there exists a $\kappa > 0$ such that $\tau_{\nu'} = \kappa \tau_\nu$. If moreover $\tau_\nu(G)$ is dense in $\R$, then actually $\nu' = \kappa \nu$.
\end{enumerate}
\end{prop}

A similar statement is true for actions semiconjugate to affine actions: let $\varphi \in \mathrm{Rep}_\mathrm{irr}(G)$ and let $\nu$ be a Radon measure on $\R$. We say that $\varphi$ \emph{preserves the projective class of} $\nu$ if for every $g \in G$ there exists $\kappa(g) > 0$ such that the measures $ \nu(\varphi(g).\cdot)$ and $\kappa(g)\nu(\cdot)$ are equal. In this case, the map
\begin{equation} \label{eq:tau}
        g \in G \mapsto \tau_\nu(g) =
        \begin{cases}
                \nu[0,\varphi(g).0) & \text{ if } 0 \leq \varphi(g).0\\
                -\nu[\varphi(g).0, 0) & \text{ if } \varphi(g).0 < 0
        \end{cases}
\end{equation}
satisfies $\tau_\nu(g_1g_2) = \tau_\nu(g_1) + \kappa(g_1)\tau_\nu(g_2)$ for all $g_1,g_2 \in G$. Any action $\varphi \in \mathrm{Rep}_\mathrm{irr}(G)$ semiconjugate to an action factoring through the affine group
\[
        \mathrm{Aff}(\R) = \{x \mapsto ax + b : a \in \R_+,\, b \in \R\}
\] preserves the projective class of the measure given by the pullback through the semiconjugacy of the Lebesgue measure on $\R$. Conversely, if $\varphi$ preserves the projective class of a Radon measure $\nu$ on $\R$, then $\varphi$ is semiconjugate to the affine action given by
\begin{equation} \label{eq:affine}
        g \in G \mapsto \left( x \in \R \mapsto \kappa(g)(x - \tau_\nu(g^{-1}))\right).
\end{equation}
where $\kappa$ is such that $\varphi(g)^\ast \nu = \kappa(g)\nu$ for all $g \in G$ and $\tau_\nu$ is defined as in \eqref{eq:tau}, see \cite[Proposition 1.2.2]{Navas2011}.

Hence any minimal action $\varphi \in \mathrm{Rep}_\mathrm{irr}(G)$ such that $\restr{\varphi}{N}$ is minimal and of type I for some normal subgroup $N \subseteq G$ is actually affine, since the normality of $N$ implies that $\varphi(G)$ preserves the projective class of any $\varphi(N)$-invariant measure on $\R$. We now verify a version of this statement in the case when $N$ is only commensurated in $G$. The starting point is the following observation.

\begin{prop} \label{prop:fixComm}
Let $\varphi \in \mathrm{Rep}_\mathrm{irr}(G)$ be a minimal action and let $H \subseteq G$ be a commensurated subgroup. If $\restr{\varphi}{H}$ has a global fixed point, then $\restr{\varphi}{H}$ is trivial.
\end{prop}
\begin{proof}
Let $g \in G$ and $h \in gHg^{-1}$, so there exists an $n \in \N_+$ such that $h^n \in H$. If $x \in \mathrm{Fix}_\varphi(H)$ then $\varphi(h^n)(x) = x$, and thus $\varphi(h)(x) = x$. This shows that $\mathrm{Fix}_\varphi(H) \subseteq \mathrm{Fix}_\varphi(gHg^{-1})$ and a symmetric argument gives $\mathrm{Fix}_\varphi(H) = \mathrm{Fix}_\varphi(gHg^{-1})$, so $\varphi(g).\mathrm{Fix}_\varphi(H) = \mathrm{Fix}_\varphi(H)$. Thus $\mathrm{Fix}_\varphi(H)$ is closed, $\varphi(G)$-invariant and nonempty, and by minimality it must be all $\R$.
\end{proof}

Given a commensurated subgroup $H \subseteq G$ and $g \in G$, we write $\triangle(g) \colon H \cap g^{-1}H g \to H$ for conjugation by $g$.

\begin{lema} \label{lema:struct1}
Let $\varphi \in \mathrm{Rep}_\mathrm{irr}(G)$ be a minimal action such that $\restr{\varphi}{H}$ is irreducible and of type I for some commensurated subgroup $H \subseteq G$. Let $\tau \colon H \to \R$ be the Conrad morphism for $H$ associated to $\restr{\varphi}{H}$.
\begin{enumerate}
        \item \label{it:s1} There exists a group morphism $\kappa \colon G \to \R^\ast_+$ such that $\tau \circ \triangle(g)(h) = \kappa(g) \tau(h)$ for all $g \in G$ and $h \in H \cap g^{-1}H g$.
        \item \label{it:s2} $\mathrm{ker}(\tau)$ is commensurated in $G$ and acts trivially through $\varphi$.
        \item \label{it:s3} If $\varphi(H)$ is non-cyclic, then $\varphi$ is conjugate to an affine action $\varphi_\mathrm{aff}$ as in \eqref{eq:affine}. Moreover, if $\varphi$ is semiconjugate to an affine action $\varphi_\mathrm{aff}$, the element $\varphi_\mathrm{aff}(h)$ is a translation whenever $h \in G$ belongs to a subgroup $L \subseteq G$ commensurate with $H$.
\end{enumerate}
\end{lema}

\begin{proof}
We first aim to find, for a fixed $g \in G$, a $\kappa(g)>0$ verifying $\tau \circ \triangle(g)(h) = \kappa(g)\tau(h)$ for all $h \in H \cap g^{-1}Hg$.

Since $\restr{\varphi}{H}$ is of type I, there exists an $H$-invariant Radon measure $\nu$ on $\R$ as in \eqref{eq:tau} such that $\tau = \tau_\nu$. Define $\nu_g$ as the Radon measure on $\R$ given by $\nu_g(\cdot) = \nu(\varphi(g)\,  \cdot)$. It is straightforward to see that the measure $\nu_g$ is $H \cap g^{-1}H g$-invariant, so there exists a $\kappa(g) > 0$ such that $\tau_{\nu_g} = \kappa(g)\tau_\nu$ on $H \cap g^{-1} H g$ by Proposition \ref{prop:conrad}. 
We conclude that when $h \in H \cap g^{-1}Hg$ is such that $\tau(h) \geq 0$, then
\[
        \tau_{\nu_g}(h) = \nu_g[\varphi(g^{-1}).x,\varphi(h)\circ \varphi(g^{-1}).x) = \nu[x, \varphi(ghg^{-1}).x) = \tau \circ \triangle(g)(h),
\]
and the same equality holds by linearity whenever $h \in H \cap g^{-1}Hg$ is such that $\tau(h) < 0$.

We now show that $\mathrm{ker}(\tau)$ is commensurated in $G$. Take $g \in G$ and set $K = \mathrm{ker}(\tau)$. Consider the map
\[
        \iota \colon h(K \cap g^{-1}Kg) \in K/(K \cap g^{-1}Kg) \mapsto h(H \cap g^{-1}Hg) \in H/(H \cap g^{-1}Hg)
\]
which is well defined since $K \subseteq H$. Notice that $H \cap g^{-1}H g \cap K \subseteq K\cap g^{-1}K g$: indeed, if $h \in H \cap g^{-1}Hg \cap K$, then
\[
        \tau(ghg^{-1}) = \tau \circ \triangle(g)(h) = \kappa(g) \tau(h) = 0,
\]
hence $h \in g^{-1}Kg$ and $h \in K \cap g^{-1}K g$. Since $H \cap g^{-1}H g \cap K \subseteq K \cap g^{-1}K g$,  the map $\iota$ is injective. Thus $K \cap g^{-1}Kg$ has finite index in $K$ and $K$ is commensurated in $G$.

By Proposition \ref{prop:conrad} the set $\mathrm{Fix}_\varphi(K)$ is non-empty. But the subgroup $K$ is commensurated in $G$, hence $\mathrm{Fix}_\varphi(K) = \R$, showing \eqref{it:s2}. By hypothesis $\restr{\varphi}{H}$ is non-trivial, so $K = \mathrm{ker}(\tau)$ does not coincide with $H$ and thus $\tau$ is non-trivial. We conclude that $\kappa \colon G \to \R^\ast_+$ is a morphism, showing \eqref{it:s1}.

We now prove \eqref{it:s3}. Again Proposition \ref{prop:conrad} shows that $\nu_g = \kappa(g)\nu$ (that is, $\varphi$ preserves the projective class of $\nu$) when $\tau_\nu(H)$ is dense. Since $\varphi(H)$ is isomorphic to $\tau(H)$ by \eqref{it:s2}, this happens if and only if $\varphi(H)$ is non-cyclic. If this is so, then $\varphi$ must be semiconjugate (actually conjugate by minimality of $\varphi$) to an affine action.

For the second statement of \eqref{it:s3}, notice that if $h \in L$ where $L \subseteq G$ is a subgroup commensurate with $H$, then there is an $n \in \N_+$ with $h^n \in H$. The action $\varphi_\mathrm{aff}$ preserves the projective class of the Lebesgue measure on $\R$, so we have a morphism $\kappa_\mathrm{Leb} \colon G \to \R^\ast_+$ such that $\varphi_\mathrm{aff}(g)$ rescales Lebesgue measure by $\kappa_\mathrm{Leb}(g)$. Thus
\[
        \kappa_\mathrm{Leb}(h)^n \mathrm{Leb} = \varphi_\mathrm{aff}(h^n)^\ast \mathrm{Leb} = \mathrm{Leb}
\]
and $\kappa_\mathrm{Leb}(h) = 1$. We conclude that $\varphi_\mathrm{aff}(h)$ preserves $\mathrm{Leb}$, that is, $\varphi_\mathrm{aff}(h)$ is a translation.
\end{proof}

\subsection{Laminar actions}\label{subsection:laminar}
Two open and bounded intervals $I,J \subseteq \R$ are said to be \emph{crossed} if $I \cap J \neq \emptyset$ and neither $I \subseteq J$ or $J \subseteq I$. A \emph{prelamination} is a collection $\cL$ of open, bounded and non-empty intervals of $\R$ that pairwise do not cross. A \emph{lamination} is a prelamination $\cL$ that is closed when seen as a subset of $\{(x,y) \in \R^2 : x < y\}$ with its natural topology, and its elements are called \emph{leaves}. An action $\varphi \colon G \to \mathrm{Homeo}_0(\R)$ is said to be \emph{laminar} if it preserves a lamination $\cL$ that is \emph{covering}, that is, that contains an increasing exhaustion of $\R$.

Laminar actions on the line are studied in \cite{BMRT}. Here we will only recall that covering laminations immediately arise in group actions on the line that admit no minimal sets. A \emph{wandering interval} for an action $\varphi \colon G \to \mathrm{Homeo}_0(\R)$ is an open bounded and non-empty interval $I \subseteq \R$ such that for all $g \in G$, either $\varphi(g).I = I$ or $\varphi(g).I \cap I = \emptyset$. An \emph{irreducible wandering interval} for $\varphi(G)$ is a wandering interval $I$ for $\varphi(H)$ such that $\mathrm{Stab}_{\varphi(H)}(I)$ acts on $I$ without global fixed points. We denote by $\cW_\varphi(G)$ the set of irreducible wandering intervals of $\varphi(G)$.

\begin{lema}[{\cite[Lemma 8.3.2]{BMRT}}] \label{lema:wandering}
Let $\varphi \in \mathrm{Rep}_\mathrm{irr}(G)$ be an action that does not admit a minimal set. Then $\cW_\varphi(G)$ is a $\varphi(G)$-invariant covering prelamination.
\end{lema}

A laminar action $\varphi \colon G \to \mathrm{Homeo}_0(\R)$ can be thought as coming from the action of $G$ on a (not necessarily simplicial) tree $\cT$ that fixes a point in the boundary of $\cT$. An analogue of the general classification of group actions on trees by J. Tits \cite{Tits1970} is still true for laminar actions, but only two cases arise when the action is assumed irreducible. Call a subset $\cX$ of a lamination $\cL$ \emph{cofinal} if for every $l \in \cL$ there exists $I \in \cX$ with $l \subseteq I$.

\begin{prop}[{\cite[Proposition 8.1.10]{BMRT}}]
Let $\varphi \in \mathrm{Rep}_\mathrm{irr}(G)$ be a laminar and irreducible action. Let $\cL$ be a covering lamination preserved by $\varphi$. Then exactly one of the following hold.
\begin{itemize}
        \item There is a cofinal set of wandering intervals for $\varphi(G)$ in $\cL$. In this case there is no minimal set for $\varphi(G)$, and for every finitely generated $H \subseteq G$ the set $\{l \in \cL : \varphi(H).l = l\}$ is cofinal in $\cL$ and $\mathrm{Fix}_\varphi(H) \subseteq \R$ is unbounded in both directions. The action $\varphi$ is said to be \emph{horocyclic}.
        \item There exists an $l \in \cL$ with cofinal $\varphi(G)$-orbit. In this case there exists a unique (non-discrete) minimal set for $\varphi(G)$ and $\varphi$ is proximal. The action $\varphi$ is said to be \emph{focal}.
\end{itemize}
\end{prop}

If $\varphi \in \mathrm{Rep}_\mathrm{irr}(G)$ is an action such that $\restr{\varphi}{N}$ admits no minimal set for some normal subgroup $N \subseteq G$ then $\cW_\varphi(N)$ is also $\varphi(G)$-invariant, and hence $\varphi$ is laminar. The same statement holds when $N$ is only commensurated.

\begin{lema} \label{lema:noMinimal}
Let $\varphi \in \mathrm{Rep}_{\mathrm{irr}}(G)$ such that $\restr{\varphi}{H}$ does not admit a minimal set for some commensurated subgroup $H \subseteq G$. Then $\cW_\varphi(H)$ is $\varphi(G)$-invariant, and hence $\varphi$ is laminar. 
\end{lema}
\begin{proof}
Since $\restr{\varphi}{H}$ does not admit a minimal set, $\restr{\varphi}{H}$ must be irreducible and Lemma \ref{lema:wandering} shows that $\cW_\varphi(H)$ is a covering prelamination. 


We claim that if $I \in \cW_\varphi(H)$ and $J \in \varphi(g).\cW_\varphi(H) = \cW_\varphi(gHg^{-1})$ for some $g \in G$, then $I$ and $J$ do not cross. Indeed, suppose towards a contradiction that $I = (a,b)$ and $J = (c,d)$ cross, so we may assume that $a < c < b < d$. The action of $\mathrm{Stab}_{\varphi(H)}(I)$ is fixed-point-free, so there exists $h \in H$ such that $\varphi(h).I = I$ and $\varphi(h).c \neq c$. If $n \in \N_+$ is such that $h^n \in g H g$, then $\varphi(h^n).J$ cannot be disjoint from $J$ because $\varphi(h^n).b = b$. Thus $\varphi(h^n).J = J$, which contradicts $\varphi(h^n).c \neq c$.  We conclude that $I$ and $J$ do not cross.

Now let $I \in \cW_\varphi(H)$, $g \in G$. We claim that $\varphi(g).I \in \cW_\varphi(H)$: take $h \in H$ and $n \in \N_+$ such that $g^{-1}h^ng \in H$. Then either $\varphi(g^{-1}h^ng).I = I$ or $\varphi(g^{-1}h^ng).I \cap I = \emptyset$. If the first option is satisfied we also have $\varphi(g^{-1}hg).I = I$. If the second option is satisfied, suppose that $\varphi(g^{-1}hg).I\cap I \neq \emptyset$. Then $\varphi(g^{-1}hg).I$ and $I$ must be crossed, contradicting the previous claim, so $\varphi(g^{-1}hg).I \cap I = \emptyset$. In any case, we conclude that $\varphi(h)\left(\varphi(g).I\right) = \varphi(g).I$ or $\varphi(h)\left(\varphi(g).I\right) \cap \varphi(g).I = \emptyset$ and that $\varphi(g).I$ is also wandering. Since $\mathrm{Stab}_{\varphi(gHg^{-1})}(\varphi(g).I)$ has no fixed points in $\varphi(g).I$ and $\varphi(gHg^{-1}) \cap \varphi(H)$ has finite index in $\varphi(H)$, it follows that $\mathrm{Stab}_{\varphi(H)}(\varphi(g).I)$ has no fixed points in $\varphi(g).I$ either. Thus $\varphi(g).I \in \cW_\varphi(H)$.


        The closure of $\cW_\varphi(H)$ in $\{(x,y) \in \R^2 : x < y\}$ gives a $\varphi(G)$-invariant covering lamination, so $\varphi$ is laminar.
\end{proof}

\section{Stability properties of $\cC$} \label{section:stability}
In this section we prove all the claimed stability properties of $\cC$ and show some examples. Subsection \ref{subsection:stability} proves Theorem \ref{teo:stability}, and Subsection \ref{subsection:BassSerre} proves Corollaries \ref{cor:stability}, along with a more general theorem for fundamental groups of graphs of groups. We also introduce the necessary definitions from Bass-Serre theory. Subsection \ref{subsection:BaumslagSolitar} considers the example of non-amenable Baumslag-Solitar groups, and shows that most of them admit uncountably many conjugacy classes of minimal type III actions on the line. 

\subsection{Proof of Theorem \ref{teo:stability}} \label{subsection:stability}
In this subsection $G$ will always denote a finitely generated group.
\begin{lema} \label{lema:noIII} Let $G$ be a finitely generated group.
Let $\varphi \in \mathrm{Rep}_\mathrm{III}(G)$ and $(f_n)_{n \geq 0}$ almost centralizing $\varphi$.  Suppose that $H \subseteq G$ is a commensurated subgroup such that $\restr{\varphi}{H}$ is non-trivial and is not of type III. Then $\lim_{n \to \infty}f_n = \mathrm{id}_\R$.
\end{lema}
\begin{proof}
First of all, $\varphi(H)$ cannot have an exceptional minimal set: suppose on the contrary that $\Lambda \subsetneq \R$ is such a minimal set and let $K \subseteq H$ be any normal finite-index subgroup. We claim that $\varphi(K)$ has a minimal set, which is $\Lambda$. Indeed, by reducing $K$ we may assume that $K$ is normal in $H$. The action of $\varphi(K)$ is cocompact since $\varphi(H)$ is cocompact and $K$ has finite index in $H$, so $\varphi(K)$ has a minimal set $\Lambda_K$. It cannot be a discrete orbit, since in that case $\varphi(H)$ would preserve a union of discrete orbits. Thus $\Lambda_K$ is unique, and by normality of $K$ in $H$ it is preserved by $\varphi(H)$. We conclude that $\Lambda_K = \Lambda$.

Now let $g \in G$, so by the previous paragraph the groups $\varphi(H), \varphi(H \cap gHg^{-1})$ and $\varphi(gHg^{-1})$ all contain a common normal finite index subgroup, hence share the same minimal set $\Lambda$. But $\varphi(g).\Lambda$ is also $\varphi(gHg^{-1})$-invariant, and hence contains a $\varphi(gHg^{-1})$-minimal set, which must be $\Lambda$ by uniqueness. We conclude that $\varphi(g).\Lambda \supseteq \Lambda$, and since $g \in G$ was arbitrary we deduce that $\Lambda$ is $\varphi(G)$-invariant. This contradicts the minimality of $\varphi$.

By Proposition \ref{prop:fixComm} we have $\restr{\varphi}{H} \in \mathrm{Rep}_\mathrm{irr}(H)$, and we will proceed to verify the conclusion according to whether $\restr{\varphi}{H}$ is of type I, II or has no minimal set.

\begin{itemize}
	\item \underline{Assume that $\restr{\varphi}{H}$ is of type II.} In this case $\restr{\varphi}{H}$  cannot have a discrete orbit, and by the previous paragraph it cannot have an exceptional minimal set either, hence it must be minimal. We say that a sequence $(h_n)_{n \geq 0} \subseteq \mathrm{Homeo}_0(\R)$ \emph{contracts} a compact interval $I \subseteq \R$ if $h_n(I) \subseteq I$ for every $n \in \N$ and $\lim_{n \to \infty} \mathrm{diam}(h_n(I)) = 0$. Thus the generator $c$ of the centralizer $\mathrm{Cent}_\varphi(H)$ such that $c(0) > 0$ can be written as
\[
	c(x) = \sup\{y \geq x : [x,y] \text{ is contracted by some sequence in }\varphi(H)\},
\]
see \cite[Section 5.2]{Ghys2001}.

For any $g \in G$ we have, by definition, that
\[
	\varphi(g) \circ c \circ \varphi(g^{-1})(x) = \sup\{y \geq x : \varphi(g^{-1}).[x,y] \text{ is contracted by some sequence in }\varphi(H)\}.
\]
Consider a sequence $(\varphi(h_n))_{n \geq 0} \subseteq \varphi(H)$ contracting an interval $[x,y]$. If $k \in \N_+$ is the index of $H \cap gHg^{-1}$ in $H$, then $(\varphi(h_n^k))_{n \geq 0}$ still contracts $[x,y]$, and every $\varphi(h_n^k)$ can be written as $\varphi(g\widetilde{h_n}g^{-1})$ for some $\widetilde{h_n} \in H$. Hence $(\widetilde{h_n})_{n \geq 0}$ contracts the interval $\varphi(g^{-1}).[x,y]$, so $ c(x) \leq \varphi(g) \circ c \circ \varphi(g^{-1})(x)$ for any $x \in \R$. A symmetric argument replacing $g$ by $g^{-1}$ shows that $c = \varphi(g) \circ c \circ \varphi(g^{-1})$ and, since $g \in G$ was arbitrary, that $\varphi$ has a non-trivial centralizer. This contradicts the fact that $\varphi$ is of type III, so $\restr{\varphi}{H}$ cannot be of type II.
\end{itemize}
We will assume that there exists an $x \in \R$ such that $(f_n^{-1}(x))_{n \geq 0}$ is unbounded and deduce a contradiction with this hypothesis in the remaining cases.  This suffices: indeed since $\varphi$ has trivial centralizer, Lemma \ref{lema:preorden} implies that the only finite accumulation point of $(f_n^{-1}(x))_{n \geq 0}$ can be $x$, hence if $(f_n^{-1}(x))_{n \geq 0}$ is bounded for every $x$ then necessarily $f_n\to \mathrm{id}_\R$. We may assume that  $\lim_{n \to \infty} f_n^{-1}(x) = \infty$. Fix a finite symmetric generating set $S \subseteq G$.

\begin{itemize}
	\item \underline{Assume that $\restr{\varphi}{H}$ is of type I.}  Let $\tau \colon H \to \R$ be a Conrad homomorphism associated to $\restr{\varphi}{H}$ and let $\kappa \colon G \to \R_+^\ast$ as in Lemma \ref{lema:struct1}. We claim that $\kappa(G) \neq \{1\}$. Suppose not, and notice that since $H \cap \bigcap_{s \in S}s^{-1}Hs$ has finite index in $H$, it contains an element $h$ such that $\tau(h) > 0$ because $\restr{\varphi}{H}$ is non-trivial. For every $s \in S$ we have
\[
        \tau(hsh^{-1}s^{-1}) = \tau(h) + \tau \circ \triangle(s)(h^{-1}) = \tau(h) - \kappa(s) \tau(h) = 0,
\]
so by Lemma \ref{lema:struct1}, \eqref{it:s2} the element $\varphi(hsh^{-1}s^{-1})$ is trivial. Hence $\varphi(h)$ is non-trivial and commutes with $\varphi(G)$, a contradiction.

Thus we can find $g \in G$ such that $\kappa(g) > 1$. Let $h \in H \cap g H g^{-1}$ such that $\varphi(h).x > \varphi(g).x$, so
\[
	U = \{\psi \in \mathrm{Rep}_\mathrm{III}(G) : \psi(h).x > \psi(g).x\}
\]
is a neighborhood of $\varphi$ in $\mathrm{Rep}_\mathrm{III}(G)$. We will show that $\varphi(h).y \leq \varphi(g).y$ for all sufficiently large $y \in \R$. Choose $r \in \N$ such that $\varphi(h^r g).s \geq \varphi(h).s$ for all $s \in [0, \varphi(h).0]$. Since $\kappa(g) > 1$ we can choose $N \in \N$ such that $\kappa(g)n - 1 \geq n + r$ for all $n \geq N$.

We claim that $\varphi(h).y < \varphi(g).y$ for all $y \geq \varphi(h^N).0$. Indeed, write $y = \varphi(h^n).t$ where $n \geq N$ and $t \in [0, \varphi(h).0]$. Denote by $\nu$ a Radon measure on $\R$ such that $\tau = \tau_\nu$ as in \eqref{eq:defTau}. Take $k=k(n,t) \in \Z$ such that
\begin{equation} \label{eq:k}
        \varphi(g^{-1}h^{k-1} g).t \leq \varphi(h^n).t < \varphi(g^{-1}h^{k}g).t. 
\end{equation}
Then
\[
	n\tau(h) = \nu[t, \varphi(h^n).t) \leq \nu[t, \varphi(g^{-1}h^{k} g).t) = \tau \circ \triangle(g^{-1})(h^{k}) = k\kappa(g^{-1})\tau(h).
\]
We conclude that $k \geq \kappa(g)n - 1 \geq n + r$ because $n \geq N$, and \eqref{eq:k} implies
\[
        \varphi(g).y = \varphi(gh^n).t \geq \varphi(h^{k-1}g).t \geq \varphi(h^{n+r}g).t > \varphi(h^{n+1}).t = \varphi(h).y
\] as desired.

Since $\lim_{n \to \infty} f_n^{-1}(x) = \infty$ we deduce that $f_n.\varphi \not \in U$ for sufficiently large $n \in \N$, which contradicts the assumption that $\lim_{n \to \infty} f_n.\varphi = \varphi$.

	\item \underline{Assume that $\restr{\varphi}{H}$ does not admit a minimal set.} By Lemma \ref{lema:noMinimal} the action $\varphi$ preserves a covering lamination $\cL$ composed of wandering intervals for $\restr{\varphi}{H}$. Since $G$ is finitely generated, the action $\varphi$ is focal. Up to conjugating $\varphi$ and the $f_n$ we may assume that for every $g \in G$ the quantity $\sup\{ \abs{\varphi(g).y - y} : y \in \R\}$ is finite by Remark \ref{rem:bounded-displacement}. 
	Fix a maximal totally ordered (for inclusion) set $\cS \subseteq \cL$, which must be closed in $\cL$ by maximality. The proofs of the following two claims follow some ideas in \cite[Section 15.1]{BMRT}.

\setcounter{claimnum}{0}
\begin{claimnum} \label{claim:stable}
        For every finite subset $F$ of $G$, there exists a leaf $l_F \in \cS$ such that for all $l \in \cS$ with $l \supseteq l_F$, we have $\varphi(F).l \subseteq \cS$.
\end{claimnum}

\begin{proof}[Proof of the claim]
Notice that since $\cL$ is a lamination and $\cS$ is maximal, if $l_1 \in \cL$ and $l_2 \in \cS$ are such that $l_1 \supseteq l_2$ then $l_1 \in \cS$ too. Thus it suffices to find a leaf $l \in \cS$ such that all its images under $\varphi(F)$ are over some element of $\cS$.

Define the finite constant
\[
        \delta_F = \sup\{\abs{\varphi(g).y- y} : g \in F, \, y \in \R\}.
\]
Since $\cS$ is maximal and $\cL$ is covering, we can find leaves $l_1, l_2 \in \cS$ such that $l_1 \supseteq l_2$ and the endpoints of $l_1$ are at distance at least $2 \delta_F$ from the endpoints of $l_2$. Thus $\varphi(g).l_1$ still contains $l_2$ for any $g \in F$, so $\varphi(F).l_1 \subseteq \cS$ and we are done.
\end{proof}

\begin{claimnum} \label{claim:endpoints}
        There exist positive constants $C,D > 0$ such that for all $y > D$, the intervals $[y - C,y]$ and $[y, y + C]$ contain the right endpoint of a leaf of $\cS$.
\end{claimnum}

\begin{proof}[Proof of the claim]
Recall that $S \subseteq G$ is a fixed finite symmetric generating subset, and let $D > 0$ be greater than the right endpoint $y_S$ of $l_S \in \cS$. Define, as in the previous claim, the finite constant
\[
        \delta_S = \sup\{\abs{\varphi(s).y - y} : s \in S, \,  y \in \R\},
\]
which must be positive since $\varphi$ is irreducible. 

For every $n \in \N$, pick $s_n \in S$ such that $\varphi(s_n s_{n-1}\cdots s_1).y_S - \varphi(s_{n-1}\cdots s_1).y_S$ is maximal. We have
\[
        \varphi(s_n s_{n-1}\cdots s_1).y_S - \varphi(s_{n-1}\cdots s_1).y_S \leq \delta_S
\]
and $\lim_{n \to \infty} \varphi(s_n \cdots s_1).y_S = \infty$ by irreducibility of $\varphi$.

Now $\varphi(s_1).l_S \in \cS$ by the definition of $l_S$, and since $\varphi(s_1).y_S > y_S$ and $\cS$ is totally ordered we see that $\varphi(s_1).l_S \supsetneq l_S$. Similarly, for every $n \in \N$ we have that $\varphi(s_n\cdots s_1).l_S$ is a leaf of $\cS$ that contains $l_S$. Thus any $y > D$ is such that $[y - \delta_S, y + \delta_S]$ contains a point $\varphi(s_m \cdots s_1).y_S$ for some $m \in \N$, which is the right endpoint of the leaf $\varphi(s_m\cdots s_1).l_S \in \cS$. Setting $C = 2 \delta_S$ and enlarging $D$ if necessary we obtain the desired conclusion.
\end{proof}

The irreducibility of $\restr{\varphi}{H}$ shows that there is an $h \in H$ with $\varphi(h).x > \varphi(s).x$ for every $s \in S$, and hence $\varphi$ belongs to the open set
\[
	U = \{\psi \in \mathrm{Rep}_\mathrm{III}(G) : \psi(h).x > \psi(s).x \text{ for every }s \in S\}.
\]
We will show that for all sufficiently large $y \in \R$ the inequality $\varphi(h).y \leq \varphi(s).y$ holds for at least one $s \in S$. Indeed, let $y_h$ be the right endpoint of the leaf $l_{\{h\} \cup S}$ from Claim \ref{claim:stable} applied to the finite set $\{h\} \cup S$. Notice that if $l \supseteq l_{\{h\} \cup S}$ then $\varphi(h).l \in \cS$, but since $\cS$ is totally ordered and $l$ is $\varphi(H)$-wandering we conclude that $\varphi(h).l = l$ and $\varphi(h)$ fixes the extremities of $l$.

 Let $C,D > 0$ be the constants from Claim \ref{claim:endpoints}, and take $y > \max(D, y_h)$. Thus the intervals $[y-C, y]$ and $[y, y+C]$ contain the right endpoints of leafs $l \supseteq l_{\{h\} \cup S}$ which must be fixed by $h$. 

Using that $\cS$ is closed we see that there are leaves $l_+ \supsetneq l_- \supseteq l_{\{h\} \cup S}$ such that $l_-$ is maximal with the property that its right endpoint $y_-$ is in $[y-C, y]$ and $l_+$ is minimal with the property that its right endpoint $y_+$ is in $[y, y+C]$. Since $\varphi$ is focal and $\varphi(S).l_- \subseteq \cS$ (because $l_- \supseteq l_{\{h\} \cup S}$) there is an $s \in S$ such that $\varphi(s).l_- \supseteq l_+$, so
\[
	\varphi(s).y \geq \varphi(s).y_- \geq y_+ = \varphi(h).y_+ \geq \varphi(h).y.
\]
We conclude that if $y$ is large enough we have $\varphi(h).y \leq \varphi(s).y$ for at least one $s \in S$.

We deduce that $f_n.\varphi \not \in U$ for sufficiently large $n \in \N$, which contradicts the assumption that $\lim_{n \to \infty} f_n.\varphi = \varphi$. \qedhere
\end{itemize}
\end{proof}
Recall that we denote by $\mathrm{I}(H)$ the isolator of a subgroup, see Subsection \ref{subsection:isolator}. The following proves Theorem  \ref{teo:stability}.
\begin{teo} \label{teo:extension}
Let $G$ be a finitely generated group. If $G$ contains a commensurated subgroup $H$ in $\cC$ such that $G/\mathrm{I}(H)$ is in $\cC$, then $G$ is in $\cC$.
\end{teo}
\begin{proof}
Let $\varphi \in \mathrm{Rep}_\mathrm{III}(G)$ and $(f_n)_{n \geq 0}$ be an almost centralizing sequence.

If $\restr{\varphi}{H}$ is trivial, then it is trivial on $\mathrm{I}(H)$ (see Lemma \ref{lema:isolator}) and $\lim_{n \to \infty} f_n.\varphi = \varphi$ inside $\mathrm{Rep}_\mathrm{III}(G/\mathrm{I}(H))$. Thus $\lim_{n \to \infty} f_n = \mathrm{id}_\R$ because $G/\mathrm{I}(H)\in \cC$. If $\restr{\varphi}{H} \in \mathrm{Rep}_\mathrm{III}(H)$, then  $(f_n)_{n \geq 0}$ almost centralizes $\restr{\varphi}{H}$ and we conclude that $\lim_{n \to \infty} f_n = \mathrm{id}_\R$ because $H \in \cC$. We conclude that $\lim_{n \to \infty} f_n = \mathrm{id}_\R$ in the remaining cases by Lemma \ref{lema:noIII}. 

We deduce that $G \in \cC$ by Theorem \ref{cor:criterio}.
\end{proof}
\begin{cor}\label{cor:extension}
    Let $G$ be a finitely generated group. If $G$ has a normal subgroup $N\in G$ such that $G/N\in \mathcal{C}$, then $G\in \cC$.
\end{cor}

Given a family of groups $(H_i)_{i\in I}$, the direct sum $\bigoplus_I H_i$ is the subgroup of $\prod_I H_i$ of all $(h_i)_{i\in I}$ such that $h_i=e_{H_i}$ for all but finitely many $i \in I$. We record that $\cC$ is stable under direct sums of finitely generated groups.

\begin{prop} \label{prop:producto}
If $(H_n)_{n \geq 0}$ are finitely generated groups, then every minimal type III action of $\bigoplus_{n \geq 0} H_n $ factors through the projection to one $H_m$. In particular if all  $H_n\in \cC$, then $\bigoplus_{n \geq 0} H_n \in \cC$.
\end{prop}
\begin{proof}

 Denote $H = \bigoplus_{n \geq 0} H_n$, and consider $\varphi \in \mathrm{Rep}_\mathrm{III}(H)$. Suppose by contradiction that $\varphi(H_m)$ and $\varphi(\bigoplus_{n\neq m} H_n)$ are both non-trivial for some $m$. Then $\varphi(H_m)$  is a non-trivial normal subgroup of $\varphi(H)$, and since it is finitely generated, it has a minimal set. Then it is  it either acts minimally, or it is a cyclic subgroup contained in the center of $\varphi(H)$, see \cite[Lemma 8.3.6]{BMRT}. The second possibility is excluded since $\varphi$ is type III and has trivial centralizer. If $\varphi(H_m)$  is minimal, then $\varphi(\bigoplus_{n\neq m} H_n)$ is contained in the centralizer of a minimal action, hence it is abelian, hence contained in the center of $\varphi(H)$, also contradicting that $\varphi$ is type III. The last assertion follows from Theorem \ref{cor:criterio}.\qedhere

\end{proof}

Recall that given groups $H, K$ and an action $K\curvearrowright X$ on a set, the permutational wreath product $H\wr_X K$ is defined as the semi-direct product 
\[\bigoplus_X H\rtimes K,\]
where $K$ acts on $\bigoplus_X H$ by shifting coordinates.  The wreath product $H\wr K$ is the permutational wreath product associated to the left-regular action of $K$ on itself.  The group $H\wr_X K$ is finitely generated provided $H, K$ are finitely generated and the action $K\curvearrowright X$ has finitely many orbits. Hence Corollary \ref{cor:extension} and Proposition \ref{prop:producto} imply the following.
\begin{cor} \label{cor:wreath}
    If $H, K\in \cC$ are finitely generated, then $H\wr_X K\in \cC$ for every action $K\curvearrowright X$ with finitely many orbits.
\end{cor}

\subsection{Fundamental groups of graphs of groups} \label{subsection:BassSerre}
For more details on the material in this subsection see \cite{Serre1977}.

For us, a \emph{graph} $\Pi$ consists in a set $V(\Pi)$ of vertices, a set $E(\Pi)$ of edges and two maps 
\[
	e \in E(\Pi) \mapsto (o(e), t(e)) \in V(\Pi) \times V(\Pi) \quad \text{ and } \quad e \in E(\Pi) \mapsto \overline{e} \in E(\Pi)
\]
where $e \mapsto \overline{e}$ is a fixed-point-free involution and $o(\overline{e}) = t(e)$ for all $e \in E(\Pi)$. The vertex $o(e)$ (resp. $t(e)$) is the \emph{origin} (resp. \emph{terminus}) of $e$, and $\overline{e}$ is the \emph{inverse edge} of $e$. A \emph{graph of groups} $(\Pi, \cG)$ is the data of a connected graph $\Pi$ and a collection $\cG$ of groups $\{G_v\}_{v \in V(\Pi)}$ (the \emph{vertex groups}) and $\{G_e\}_{e \in E(\Pi)}$ (the \emph{edge groups}) with $G_e = G_{\overline{e}}$ for all $e \in E(\Pi)$, equipped with injective morphisms $\iota_{e,o(e)}\colon G_e \to G_{o(e)}$. The \emph{fundamental group} $\pi_1(\Pi, \cG)$ of the graph of groups $(\Pi, \cG)$ is defined as follows: choose a connected subgraph $T \subseteq \Pi$ with no cycles and that contains all vertices of $\Pi$. 
Denote by $t_e$ a symbol indexed by $e \in E(\Pi)$. Then $\pi_1(\Pi, \cG)$ is the quotient of the free product $\bigast_{v \in V(\Pi)} G_v \ast \bigast_{e \in E(\Pi \setminus T)} t_e$ by the relations 
\[
	t_e \iota_{e, o(e)}(g) t_e^{-1} = \iota_{\overline{e}, o(\overline{e})}(g) \quad \text{and} \quad t_e = t_{\overline{e}}^{-1}
\]
for all $e \in E(\Pi \setminus T)$ and $g \in G_{t(e)}$, and $\iota_{e, o(e)}(g) = \iota_{e,t(e)}(g)$ for all $e \in E(T)$ and $g \in G_{e}$. HNN-extensions and amalgamated products correspond to the case when $\Pi$ is a loop or a single edge, respectively. As in these special cases, the natural maps $G_v \to \pi_1(\Pi, \cG)$ and $G_e \to \pi_1(\Pi, \cG)$ are always injective \cite[Théorème 13]{Serre1977}, and 
one sees that $\pi_1(\Pi, \cG)$ is isomorphic to 
\[
	\langle \langle G_v,\, v \in V(\Pi)\rangle \rangle \rtimes \langle t_e,\,  e \in E(\Pi \setminus T)\rangle
\]
where $\langle t_e,\,  e \in E(\Pi \setminus T)\rangle \cong F_{b(\Pi)}$ is free of rank $b(\Pi)$, the first Betti number of the graph $\Pi$. The group $\pi_1(\Pi, \cG)$ does not depend, up to isomorphism, on the choice of $T$ \cite[Proposition 20]{Serre1977}.

Bass-Serre theory \cite[\textsection 5]{Serre1977} shows that presentations of a group $G$ as a (finite) fundamental group of a graph of groups are in correspondence with (cocompact) actions of $G$ on trees $\cT$ by tree automorphisms with no inversions, that is, such that there is no $g \in G$ and $e \in E(\cT)$ with $g.e = \overline{e}$. One direction of this correspondence takes a graph of groups $(\Pi, \cG)$ and produces a canonically defined action of $\pi_1(\Pi, \cG)$ on a tree $\cT_{\Pi, \cG}$, called the \emph{Bass-Serre tree} of $(\Pi, \cG)$, such that $\Pi = G\backslash \cT_{\Pi, \cG}$. The conjugates of the vertex groups $G_v$ (resp. of the edge groups $G_e$) in $\pi_1(\Pi, \cG)$ are exactly the stabilizers of vertices (resp. edges) that project to $v$ (resp. $e$) in $\Pi$, and the maps $\iota_{e, o(e)}$ are induced by the inclusion of edge stabilizers into vertex stabilizers. Moreover, the tree $\cT_{\Pi, \cG}$ is locally finite if and only if $\iota_{e, o(e)}(G_e)$ has finite index in $G_{o(e)}$ for all $e \in E(\Pi)$.

The following lemma is well known, and we include it for completeness.

\begin{lema}\label{lema:commensurated-graphs}
Let $(\Pi, \cG)$ be a graph of groups where all inclusions $\iota_{e, o(e)},\, e \in E(\Pi)$ have finite-index image in $G_{o(e)}$. Then any two stabilizers of vertices or edges in the Bass-Serre tree $\cT_{\Pi, \cG}$ are commensurate. In particular, any vertex group or edge group in $\pi_1(\Pi, \cG)$ is commensurated in $\pi_1(\Pi, \cG)$.

If $H$ is any vertex group or edge group, then  $\mathrm{I}(H)=\langle \langle G_v,\, v \in V(\Pi)\rangle \rangle$, and thus $G/\mathrm{I}(H)$ is free of rank $b(\Pi)$.
\end{lema}

\begin{proof}
Write $G = \pi_1(\Pi, \cG)$. The tree $\cT_{\Pi, \cG}$ is locally finite, so its balls for the natural path metric $d$ are finite. If $u,u'$ are vertices in $\cT_{\Pi, \cG}$, then $\mathrm{Stab}_G(u) \cap \mathrm{Stab}_G(u')$ has index at most $\abs{\{w \in \cT_{\Pi, \cG} : d(u,w)= d(u,u')\}}$ inside $\mathrm{Stab}_G(u)$. The same is true if $u$ or $u'$ are edges in $\cT_{\Pi, \cG}$ since edge groups have finite index in vertex groups.
If $H$ is any vertex group or edge group, then  $\mathrm{I}(H)$ is normal and contains all vertex groups (see Lemma \ref{lema:isolator}), so contains $\langle \langle G_v,\, v \in V(\Pi)\rangle \rangle$, and we actually actually have equality since  $G/\langle \langle G_v,\, v \in V(\Pi)\rangle \rangle=F_{b(\Pi)}$ is torsion-free. \qedhere

\end{proof}


\begin{cor}\label{cor:BassSerre}
Suppose $G$ is a finitely generated group that can be presented as the fundamental group $\pi_1(\Pi, \cG)$ of a finite graph of groups $(\Pi, \cG)$ where all the inclusions $\iota_{e, o(e)},\, e \in E(\Pi)$ have finite-index image in $G_{o(e)}$ and every $G_v, \, v \in V(\Pi)$ is in $\cC$. Then $G \in \cC$ if and only if $b(\Pi)$ is at most 1.
\end{cor}
\begin{proof}
Since the non-abelian free group $F_m$ is not in $\cC$ for $m\ge 2$ and the class $\cC$ is obviously closed under quotient, we have $b(\Pi) \leq 1$ if $G \in \cC$. The converse holds by Theorem \ref{teo:extension} and  Lemma \ref{lema:commensurated-graphs}.
\end{proof}

\begin{ejm}
Recall that a group $G$ is a \emph{generalized Baumslag-Solitar group of rank $n \in \N_+$} if it is the fundamental group of a finite graph of groups where all vertex and edge groups are isomorphic to $\Z^n$.  Corollary \ref{cor:BassSerre} characterizes exactly which generalized Baumslag-Solitar groups belong to $\cC$.

The class of rank 2 generalized Baumslag-Solitar groups contains the examples of Leary-Minasyan of the first CAT(0) and non-virtually biautomatic groups \cite{LearyMinasyan2021}, and the rank 1 generalized Baumslag-Solitar groups has been previously examined from geometric and algebraic perspectives as a natural generalization of the Baumslag-Solitar groups, see for instance \cite{Whyte2001, levittAutomorphisms}.
\end{ejm}

\begin{proof}[Proof of Corollary \ref{cor:stability}]
Corollary \ref{cor:extension} implies \eqref{i-stability-extension},  Corollary \ref{cor:wreath} implies  \eqref{i-stability-wreath}, and Corollary \ref{cor:BassSerre} implies \eqref{i-stability-HNN} and \eqref{i-stability-amalgamated}. \qedhere
\end{proof}
%

\subsection{Producing many non-conjugate actions} \label{subsection:BaumslagSolitar}
Even though Theorem \ref{teo:stability} allows to prove that many groups $G$ are in $\cC$,  it does not produce an explicit parametrization of  minimal  actions of $G$ on the line up to conjugacy, even when one understands very well such actions for the groups $H$ and $G/\mathrm{I}(H)$ in its statement. For instance, consider a non-amenable Baumslag-Solitar group
\[G = \mathrm{BS}(m,n) = \langle g,t \mid gt^ng^{-1} = t^m\rangle, \, 2 \leq m < n.\]
Then $G$  fits in the setting of Theorem \ref{teo:stability} in a very simple way, since $H=\langle t \rangle\simeq \Z$ is commensurated and $G/\mathrm{I}(H)\simeq \Z$. We will show that (in contrast with the solvable case, see \cite{Rivas2010}) there are many non-conjugate minimal actions of $G$ on the line. We do not know an explicit classification of all its  minimal actions up to conjugacy.

\begin{lema} \label{lema:BS}
    Let $G$ be a finitely generated group and let $(\varphi_s)_{s \in [0,1]} \subseteq \mathrm{Rep}_\mathrm{irr}(G)$ be a continuous path of actions that are not of type I such that $\varphi_0$ is not semiconjugate to $\varphi_1$. Then there exist uncountably many actions in $\mathrm{Rep}_\mathrm{irr}(G)$ that are pairwise non-semiconjugate and that are not of type I. 
\end{lema}
\begin{proof}
Denote by $r \colon \mathrm{Rep}_\mathrm{irr}(G) \to \mathrm{Harm}(G)$  the continuous retraction from Theorem \ref{teo:existenciaDeroin}. Up to post-composing by $r$ we may assume that $(\varphi_s)_{s \in [0,1]}$ is contained in $\mathrm{Harm}(G)$. Since $\Psi$ restricted to $X = \mathrm{Harm}(G) \setminus \mathrm{Fix}_\Psi(\mathrm{Harm}(G))$ is locally free, by \cite[Lemma 5.1]{deroinHurtado} every $\varphi \in X$ admits a \emph{flow box}: that is, there is a Polish space $T$, a neighborhood $U \subseteq X$ of $\varphi$ and a homeomorphism $\alpha \colon U \to T \times (-1, 1)$ conjugating $\Psi$ with the vertical flow
\[
    (\tau, x) \in T \times (-1, 1) \mapsto s.(\tau, x) = (\tau, x+s) \in T \times (-1,1)
\]
for small $s \in \R$.

Define
\[
    \delta = \sup\{s \in [0,1] : \varphi_s \in \mathrm{Orb}_\Psi(\varphi_0)\},
\]
which belongs to $[0,1)$ since $\varphi_1$ is not in the $\Psi$-orbit of $\varphi_0$. 
Let $U \subseteq X$ be a neighborhood of $\varphi_\delta$ so that there exists a flow box $\alpha \colon U \to T \times (-1,1)$ for some Polish space $T$, and call $\alpha_T$ the composition of $\alpha$ with the projection to the $T$-factor. Find $\delta' > \delta$ so that $\varphi_s$ remains in $U$ for $s \in [\delta, \delta']$. The map $\alpha_T \circ \varphi \colon [\delta, \delta'] \to T$ is continuous and not constant by the choice of $\delta$, its image in the transversal $T$ is uncountable. Since every $\Psi$-orbit in $X$ intersects $U$ in an at most countable number of connected components, which are fibers of $\alpha_T$ themselves, the actions $(\varphi_s)_{s \in [\delta, \delta']}$ must intersect uncountably many $\Psi$-orbits.
\end{proof}

\begin{prop}
There exist uncountably many actions of $G = \mathrm{BS}(m,n),\, 2 \leq m < n$ on the line that are pairwise non-semiconjugate, minimal and proximal.
\end{prop}
\begin{proof}
Notice first that if $\varphi \in \mathrm{Rep}_\mathrm{irr}(G)$ is such that $\varphi(t)$ has no fixed points, the relation $gt^mg^{-1} = t^n$ implies that $\varphi_\mathrm{FF}(g)$ has a compact non-empty set of fixed points. This shows that $\varphi$ cannot be semiconjugate to a cyclic action or to a minimal type II action. Since the abelianization of $G$ is cyclic it cannot be semiconjugate to a minimal type I action, so it must be of type III.

Farb and Franks \cite[Theorem 1.5]{FarbFranks2020} construct a faithful action $\varphi_\mathrm{FF} \in \mathrm{Rep}_\mathrm{irr}(G)$ by analytic diffeomorphisms where $\varphi_\mathrm{FF}(t)$ is the translation $x \mapsto x + 1$. In particular $\varphi_\mathrm{FF}$ cannot be of type I by the previous paragraph. Moreover, the minimalization $r(\varphi_\mathrm{FF})$ of $\varphi_\mathrm{FF}$ is faithful: indeed if $r(\varphi_\mathrm{FF})(h)=\mathrm{id}_\R$ for some $h\in G$, then the analytic diffeomorphism $\varphi_\mathrm{FF}(h)$ fixes pointwise the minimal set of $\varphi_\mathrm{FF}$. This set cannot be discrete since $\varphi_\mathrm{FF}$ is not type I, hence $h=e_G$.

The affine action $\varphi_\mathrm{aff}$ of $G$ defined by setting  $\varphi_\mathrm{aff}(g).x = \frac{n}{m}x$ and $\varphi_\mathrm{aff}(t).x = x + 1$ for all $x \in \R$ is minimal and not faithful (because $G$ is non-amenable), so it is not conjugate to $r(\varphi_\mathrm{FF})$. Define a path of proximal actions $(\varphi_s)_{s \in [0,1]} \subseteq \mathrm{Rep}_\mathrm{irr}(G)$ as follows: choose a continuous path of orientation-preserving homeomorphisms $\psi \colon [0,1] \to \mathrm{Homeo}_0([0,m],[0,n])$ such that $\psi_0 = \restr{\varphi_ \mathrm{FF}(g)}{[0,n]}$ and $\psi_{1} = \restr{\varphi_\mathrm{aff}(g)}{[0,n]}$. For $s \in [0, 1]$, define $\varphi_s \in \mathrm{Rep}_\mathrm{irr}(G)$ by
\[
    \varphi_s(t) = \varphi_\mathrm{aff}(t)\quad \text{ and } \quad \varphi_s(g) = \widetilde{\psi_s}
\] where $\widetilde{\psi_s}$ is the extension of $\psi_s$ to $\R$ that satisfies $\widetilde{\psi_s}(x + n) = \widetilde{\psi_s}(x) + m$ for all $x \in \R$. The path $(\varphi_s)_{s \in [0,1]}$ verifies the hypotheses of Lemma \ref{lema:BS}, implying the desired conclusion.
\end{proof}

\section{Micro-supported and piecewise projective groups} \label{subsection:micro}
This section proves Theorem \ref{teo:PL} by showing a general result (Proposition \ref{prop:micro}) for micro-supported groups acting minimally on $\R$ and then specializing to piecewise projective groups. In this section, whenever given a group $G \subseteq \mathrm{Homeo}_0(\R)$ we will call this distinguished action on $\R$ the \emph{standard action} of $G$ and denote it by $(g,x) \in G \times \R \mapsto g.x$. For simplicity of exposition in the proofs we will use Theorem \ref{cor:criterio} for sequences of translations instead of sequences of homeomorphisms.

A subgroup $G \subseteq \mathrm{Homeo}_0(\R)$ is said to be \emph{micro-supported} if for every relatively compact interval $I \subseteq \R$ the subgroup
\[
	G_I = \{g \in G : g(x) = x \text{ for all }x \in \R \setminus I \}
\]
is non-trivial. Notice that specifying a micro-supported group also implies fixing an embedding $G \subseteq \mathrm{Homeo}_0(\R)$.

\begin{prop}[{\cite[Chapter 3]{BMRT}}] \label{prop:BMRTmicro}
Let $G \subseteq \mathrm{Homeo}_0(\R)$ be a countable group acting minimally on $\R$. Then $G$ is micro-supported if and only if it contains an element of relatively compact support.

In this case, $G$ admits a maximal normal subgroup $[G_c, G_c]$ where $G_c$ is the subgroup of compactly supported elements of $G$ for the standard action.
\end{prop}

One of the main theorems of \cite{BMRT} says that any faithful minimal action of a micro-supported $G$ that is not itself micro-supported is necessarily laminar. We need a more detailed version of this statement making explicit some conditions on the resulting lamination, and some more notation.

\begin{defn}
Given a group $G \subseteq \mathrm{Homeo}_0(\R)$, we say that a (possibly unbounded) non-empty interval $I \subsetneq \R$ is \emph{$G$-good} if, by denoting $\cO_I = \{g.I : g \in G\}$, the following properties are verified:
\begin{itemize}
	\item for every $I_1, I_2 \in \cO_I$ with $I_1 \cap I_2 \neq \emptyset$, either $I_1 \subseteq I_2$ or $I_2 \subseteq I_1$, and
	\item for every $I_1, I_2 \in \cO_I$ there exists $I_3 \in \cO_I$ such that $I_3 \supseteq I_1 \cup I_2$.
\end{itemize}
If $\varphi \in \mathrm{Rep}_\mathrm{III}(G)$, we say that $I$ \emph{generates a lamination for $\varphi$} if one of the following properties is verified.
\begin{itemize}
	\item Either for every (equivalently, for some) $K \in \cO_I$ the support $\mathrm{supp}_\varphi([G_K, G_K])$ has bounded connected components, or
	\item for every (equivalently, for some) $K \in \cO_I$ the group $\varphi([G_K, G_K])$ acts on $\R$ with no minimal set.
\end{itemize}
\end{defn}

If $I$ is a $G$-good interval generating a lamination for $\varphi \in \mathrm{Rep}_\mathrm{III}(G)$, the wandering intervals $\bigcup_{K \in \cO_I} \cW_\varphi([G_K, G_K])$ define a $\varphi$-invariant covering prelamination \cite[Proposition 8.3.4]{BMRT}.

\begin{lema}[{\cite[Section 9.1]{BMRT}}] \label{lema:BMRT}
	Let $G \subseteq \mathrm{Homeo}_0(\R)$ be a micro-supported finitely generated group acting minimally on $\R$, and let $\varphi \in \mathrm{Rep}_\mathrm{III}(G)$ be a faithful action that is not conjugate to the standard action. Then there exists a $G$-good interval $I \subsetneq \R$ such that $I$ generates a lamination for $\varphi$.
\end{lema}

\begin{lema} \label{lema:laminarGgood}
Let $G \subseteq \mathrm{Homeo}_0(\R)$ be a finitely generated group acting irreducibly on $\R$, let $\varphi \in \mathrm{Rep}_\mathrm{III}(G) \cap \mathrm{Harm}(G)$ and let $(t_n)_{n \geq 0} \subseteq \R$ such that $\lim_{n \to \infty} \Psi^{t_n}(\varphi) = \varphi$. If there is a $G$-good interval that generates a lamination for $\varphi$, then $\lim_{n \to \infty} t_n = 0$.
\end{lema}
\begin{proof}
Fix $S \subseteq G$ a finite symmetric generating set. Let
\[
	\delta_S = \sup\{\abs{\varphi(s).y - y} : s \in S,\,  y \in \R\},
\]
which is finite since $\varphi \in \mathrm{Harm}(G)$ (see Remark \ref{rem:bounded-displacement}). Let $I \subsetneq \R$ be a $G$-good interval such that
\[
	\cL_\varphi = \bigcup_{K \in \cO_I} \cW_\varphi([G_K, G_K])
\]
is a covering prelamination.  Since $I$ is $G$-good, the set $\bigcup_{K \in \cO_I}[G_K, G_K]$ is a normal subgroup of $G$, whose image through $\varphi$ acts irreducibly on $\R$. Hence there exists $\widetilde{K} \in \cO_I$, $x \in \widetilde{K}$ and $g \in [G_{\widetilde{K}}, G_{\widetilde{K}}]$ such that $\abs{\varphi(g).x - x} > \delta_S$. By enlarging $\widetilde{K}$ if necessary, we may assume that $s.\widetilde{K}\cap\widetilde{K} \neq \emptyset$ for all $s \in S$.

\setcounter{claimnum}{0}
\begin{claimnum} \label{claim:JK}
	For any $J \in \cL_\varphi$ there exists $K \in \cO_I$ with $K \supseteq \widetilde{K}$ such that $\cW_\varphi([G_K, G_K])$ contains an interval $\hat{J} \supseteq J$.
\end{claimnum}
\begin{proof}[Proof of the claim]
Suppose first that the $\varphi([G_K, G_K]),\, K \in \cO_I$ act without a minimal set. Then $\cW_\varphi([G_{\widetilde{K}}, G_{\widetilde{K}}])$ is a covering prelamination for $\varphi([G_{\widetilde{K}}, G_{\widetilde{K}}])$ by Lemma \ref{lema:wandering}, so we may take $K = \widetilde{K}$.

If all the supports $\mathrm{supp}_\varphi([G_K, G_K]),\, K \in \cO_I$ have a bounded connected component, since $\cL_\varphi$ is a covering prelamination we may find $K' \in \cO_I$ such that there is a connected component $J'$ of $\varphi([G_{K'}, G_{K'}])$ containing $J$. Since $I$ is $G$-good, there is a $K \in \cO_I$ with $K \supseteq K'\cup \widetilde{K}$. If $\hat{J}$ is the connected component of $\mathrm{supp}_\varphi([G_K, G_K])$ containing $J'$ we are done.
\end{proof}

Now fix $\widetilde{J} \in \cL_\varphi$ such that $\varphi(s).\widetilde{J} \cap \widetilde{J} \neq \emptyset$ for all $s \in S$ and $x, \varphi(g).x \in \widetilde{J}$, so $\varphi(g).\widetilde{J} = \widetilde{J}$. Consider an arbitrary $J \in \cL_\varphi$ containing $\widetilde{J}$ and take intervals $K \in \cO_I,\, \hat{J} \supseteq J$ supplied by Claim \ref{claim:JK} applied to $J$. We say that a (finite or infinite) sequence $(s_n)_{n \geq 0} \subseteq S$ is \emph{$K$-admissible} if $(s_n\cdots s_1).K \supseteq K$ for all $n \geq 0$. Notice that when $(s_k \cdots s_1).K \supseteq K$, then the interval $\varphi(s_k \cdots s_1).\hat{J}$ is wandering for $\varphi([G_{K}, G_{K}])$ and is thus fixed by $\varphi(g)$ whenever $\varphi(s_k \cdots s_1).\hat{J} \supseteq \widetilde{J}$.
\begin{claimnum}
	Either there exists a $K$-admissible sequence $(s_n)_{n \geq 0}$ such that the intervals $\varphi(s_n\cdots s_1).\hat{J}$, $n \in \N$ are unbounded, or there exists a $K$-admissible sequence $(s_n)_{n \geq 0}$ and $m \in \N$ such that $\varphi(s_m \cdots s_1).\hat{J} \subsetneq \widetilde{J}$.
\end{claimnum}
\begin{proof}
Towards a contradiction, suppose that for every $K$-admissible sequence $(s_n)_{n \geq 0}$ the intervals $\varphi(s_n \cdots s_1).\hat{J}$ remain bounded and never lie inside $\widetilde{J}$. Let $J_M \in \cL_\varphi$ such that $\varphi(s_n \cdots s_1).\hat{J} \subseteq J_M$ for all $n \in \N$. By irreducibility of $\varphi$, we can find a finite sequence $d_1, \ldots, d_k \in S$ such that $\varphi(d_k \cdots d_1).J_M \subsetneq \widetilde{J}$ and $\varphi(d_j \cdots d_1).J_M \supseteq \widetilde{J}$ for all $0 \leq j < k$. By irreducibility of the standard action of $G$, we may find a finite sequence $e_1, \ldots, e_l \in S$  such that the sequence $e_1,\ldots, e_l, d_1,\ldots, d_k$ is $K$-admissible. But
\[
	\varphi(d_k \cdots d_1 e_l \cdots e_1). \hat{J} \subseteq \varphi(d_k \cdots d_1).J_M \subsetneq \widetilde{J},
\]
a contradiction.
\end{proof}

If for some $J \supseteq \widetilde{J}$ the intervals $K, \hat{J}$ verify the first option of the previous claim, then we conclude that $\mathrm{Fix}(\varphi(g))$ is $\delta_S$-dense outside of $\hat{J}$. If for every $J \supseteq \widetilde{J}$ the intervals $K, \hat{J}$ verify the second option, then again we conclude that $\mathrm{Fix}(\varphi(g))$ is $\delta_S$-dense outside of $\widetilde{J}$. In any case, $t_n - x$ cannot exit a bounded interval around $x, \varphi(g).x$ since $\abs{\Psi^{t_n}(\varphi)(g).x - x} > \delta_S$ for $n$ large enough. Thus $(t_n)_{n \geq 0}$ is bounded and $\lim_{n \to \infty} t_n = 0$ again by Theorem \ref{cor:criterio}.
\end{proof}

\begin{prop} \label{prop:micro}
	Let $G \subseteq \mathrm{Homeo}_0(\R)$ be a micro-supported finitely generated group acting minimally on $\R$. Then $G \in \cC$ if and only if $G/[G_c, G_c] \in \cC$.
\end{prop}
\begin{proof}
One implication is clear, so suppose that $G/[G_c, G_c] \in \cC$. Consider $\varphi \in \mathrm{Rep}_\mathrm{III}(G) \cap \mathrm{Harm}(G)$ and $(t_n)_{n \geq 0} \subseteq \R$ such that $\lim_{n \to \infty} \Psi^{t_n}(\varphi) = \varphi$. Recall that the action $\varphi$ is either faithful or factors through $G/[G_c, G_c]$, so we will prove the proposition by assuming that $\varphi$ is faithful and showing that $\lim_{n \to \infty} t_n = 0$ in this case.

Suppose first that $\varphi$ is micro-supported. Fix an open bounded interval $I \subseteq \R$, take a non-trivial element $g \in G$ with $\mathrm{supp}_\varphi(g) \subseteq I$ and let $x \in I$ such that $\varphi(g).x \neq x$. But if $x-t_n \not \in I$ we have $\Psi^{t_n}(\varphi)(g).x = x$, so $x - t_n \in I$ from some $n \in \N$ onwards and hence $(t_n)_{n \geq 0}$ is bounded. By Theorem \ref{cor:criterio} we conclude that $\lim_{n \to \infty} t_n = 0$.

Now suppose that $\varphi$ is not micro-supported, so $\varphi$ is not semiconjugate to the standard action. By Lemma \ref{lema:BMRT} there is a $G$-good interval that generates a lamination for $\varphi$, and Lemma \ref{lema:laminarGgood} shows that $\lim_{n \to \infty} t_n = 0$. We conclude that $G \in \cC$ by Theorem \ref{cor:criterio}.
\end{proof}

\begin{ejm}
Proposition \ref{prop:micro} is already enough to supply many new examples of elementary amenable subgroups of $\mathrm{Homeo}_0(\R)$ that are in $\cC$ and are not virtually solvable, for instance the Brin-Navas group $B$ defined in \cite{Navas2004, Brin2005}. The group $B$ can be defined as follows: fix a non-empty open bounded interval $(x,y) \subseteq \R$ and $f \in \mathrm{Homeo}_0(\R)$ such that the only fixed point of $f$ is a single point in $(x,y)$ and such that $f(x) < x < y < f(y)$. Choose a non-decreasing homeomorphism $w_0 \in \mathrm{Homeo}_0(\R)$ whose support is $(x,y)$ and such that $w_0(f^{-1}(x)) = f^{-1}(y)$. Then $B$ is the group generated by $f, w_0$, and its defining action on $\R$ is minimal and micro-supported. It is clear that these conditions can be met by piecewise affine homeomorphisms $f, w_0$.

The conditions on $w_0$ imply that $w_0$ and $w_1 = f w_0 f^{-1}$ are such that $w_0$ and $w_1 w_0 w_1^{-1}$ have disjoint support, so $w_0, w_1$ generate a wreath product $\Z \wr \Z$. Similarly, the subgroup generated by any finite subset of the $w_n = f^n w_0 f^{-n}, n \in \Z$ is an iterated wreath product $(\cdots ((\Z \wr \Z) \wr \Z) \cdots )\wr \Z$. Thus $B$ is non-virtually solvable and elementary amenable. See \cite{BleakBrinMoore2021} for more examples of this kind.
\end{ejm}

\begin{rmk}
Suppose $G \subseteq \mathrm{Homeo}_0(\R)$ is a group acting minimally on $\R$ that admits a generating set $\{s_1, s_2, \ldots, s_k\}$ where $s_1$ is supported on an interval $(-\infty, x)$, $s_k$ is supported on an interval $(y, \infty)$ and the $s_2, \ldots, s_{k-1}$ have relatively compact support. Then $G$ is micro-supported and writing $G/[G_c, G_c]$ as the extension \eqref{eq:extension} shows that $G/[G_c, G_c]$ is solvable, so by Proposition \ref{prop:micro} we have $G \in \cC$. This applies to the family of pre-chain groups studied in \cite{KimKoberdaLodha2019}, who were already known to lie in $\cC$ from \cite[Section 16.3.3]{BMRT}.
\end{rmk}

\begin{lema}[{\cite[Lemma 9.1.2]{BMRT}}] \label{lema:minimalProducto}
Let $\varphi \in \mathrm{Rep}_\mathrm{irr}(G)$ and $H_1, H_2 \subseteq G$ two commuting subgroups such that $\varphi(H_1)$ admits a minimal set $\Lambda$. Then either $[H_1, H_1]$ or $[H_2, H_2]$ fixes $\Lambda$ pointwise.
\end{lema}

\begin{prop} \label{prop:irreduciblePL}
Let $G \subseteq \mathrm{PProj}_0(\R)$ be a finitely generated group acting irreducibly on $\R$. Then $G \in \cC$.
\end{prop}
\begin{proof}
Let $\Lambda \subseteq \R$ be a minimal set for the standard action of $G$. Let $N \subseteq G$ the kernel of the action on $\Lambda$, which is composed of compactly supported homeomorphisms.

\setcounter{claimnum}{0}
\begin{claimnum}
The group $G/[N,N]$ belongs to $\cC$. 
\end{claimnum}
\begin{proof}[Proof of the claim]
Set $Q = G/N$. We first prove that $Q \in \cC$. Notice that $Q$ is either abelian (and the conclusion follows) or admits a minimal faithful action on $\R$ which is a continuous factor of the standard action of $G$. In this case, the groups of germs of $Q$ at $-\infty, \infty$ are isomorphic to those of $G$ and both are 2-step solvable. Unless $Q$ is 2-step solvable (and the conclusion follows from Proposition \ref{teo:extension}), there exists $g_1, g_2, g_3, g_4 \in Q$ such that $[[g_1, g_2], [g_3, g_4]]$ is a non-trivial homeomorphism of relatively compact support. By Proposition \ref{prop:BMRTmicro} this implies that $Q$ is micro-supported, and Proposition \ref{prop:micro} shows that $Q \in \cC$ if and only if $Q/[Q_c, Q_c]$. But $Q/[Q_c, Q_c]$ is finitely generated and can be written as the extension
\begin{equation} \label{eq:extension}
	1 \longrightarrow Q_c/[Q_c, Q_c] \longrightarrow Q/[Q_c, Q_c] \longrightarrow Q/Q_c \longrightarrow 1
\end{equation}
where $Q_c/[Q_c, Q_c]$ is abelian and $Q/Q_c$ solvable since it is contained in the product of the groups of germs of $Q$ at $-\infty, \infty$, which are subgroups of $\mathrm{Aff}(\R)$. We conclude that $Q/[Q_c, Q_c] \in \cC$ by Proposition \ref{teo:extension}. Hence $Q \in \cC$ in any case.

By writing the finitely generated group $G/[N,N]$ as an extension of $Q$ by the abelian group $N/[N,N]$ we conclude that $G/[N,N] \in \cC$ by Proposition \ref{teo:extension}.
\end{proof}

Now let $\varphi \in \mathrm{Rep}_\mathrm{III}(G) \cap \mathrm{Harm}(G)$ and $(t_n)_{n \geq 0} \subseteq \R$ such that $\lim_{n \to \infty} \Psi^{t_n}(\varphi) = \varphi$. If $\varphi([N,N])$ is trivial, then $\lim_{n \to \infty} t_n = 0$ because $G/[N,N] \in \cC$. Since $N$ is composed of compactly supported homeomorphisms, we may assume that $\varphi([N_{(x, \infty)}, N_{(x, \infty)}])$ and $\varphi([N_{(-\infty, x)}, N_{(-\infty, x)}])$ are non-trivial for every $x \in \R$. 

Fix $x \in \R$.
\begin{claimnum}
One of the $G$-good intervals $(x, \infty), (-\infty, x)$ generates a lamination for $\varphi$.
\end{claimnum}
\begin{proof}[Proof of the claim]
Suppose first that $\varphi([G_{(x, \infty)}, G_{(x, \infty)}])$ has a global fixed point. If the support $\mathrm{supp}_\varphi([G_{(x, \infty)}, G_{(x, \infty)}])$ has a bounded connected component $C$, then the claim follows. If not, then one of the connected components of $\mathrm{supp}_\varphi([G_{(x, \infty)}, G_{(x, \infty)}])$ is an interval of the form $(h(x), \infty)$ or $(-\infty, h(x))$, and we may assume that it is of the form $(h(x),\infty)$ (the other case is analogous). Define a map $h$ from $\mathrm{Orb}_\mathrm{st}(x)$, the orbit of $x$ under the standard action of $G$, to $\R$ by setting
\[
	y \in \mathrm{Orb}_\mathrm{st}(x) \mapsto \inf \mathrm{supp}_\varphi([G_{(y, \infty)}, G_{(y, \infty)}]).
\]
The map $h$ intertwines the standard action of $G$ with $\varphi$ and is monotone. Up to replacing it by $x \mapsto -h(x)$ it extends to a semiconjugacy between $\varphi$ and the standard action of $G$. Since $\varphi$ is minimal, $h$ is continuous. Hence the standard action of $G$ does not have a discrete orbit and $\varphi$ is conjugate to the minimal action of $G/N$ which is a factor of the standard action of $G$. This is a contradiction, since $\varphi([N,N])$ is non-trivial.

By repeating the same argument as the previous paragraph with $\varphi([G_{(-\infty, x)}, G_{(-\infty, x)}])$, we may suppose then that the groups $\varphi([G_{(x, \infty)}, G_{(x, \infty)}]),  \varphi([G_{(-\infty, x)}, G_{(-\infty, x)}])$ act irreducibly. Suppose first that $\varphi(G_{(x, \infty)})$ admits a minimal set $\Lambda \subseteq \R$. Since the groups $\varphi(G_{(x, \infty)}), \varphi(G_{(-\infty, x)})$ commute, Lemma \ref{lema:minimalProducto} shows that one of the $\varphi([G_{(-\infty, x)}, G_{(-\infty, x)}]), \varphi([G_{(x, \infty)}, G_{(x, \infty)}])$ must fix $\Lambda$ pointwise, contradicting the irreducibility of their action. 

Suppose instead that $\varphi(G_{(x, \infty)})$ admits no minimal set, so its action on $\R$ is irreducible and not cocompact. Thus the action of $\varphi([G_{(x, \infty)}, G_{(x, \infty)}])$ is irreducible and not cocompact either. We deduce that $\varphi([G_{(x, \infty)}, G_{(x, \infty)}])$ admits no minimal set, which impĺies the claim in this case.
\end{proof}

By Lemma \ref{lema:laminarGgood} we conclude that $\lim_{n \to \infty} t_n = 0$. Hence $G \in \cC$ by Theorem \ref{cor:criterio}.
\end{proof}

\begin{cor}[Theorem \ref{teo:PL}] \label{cor:PL}
Let $G \subseteq \mathrm{PProj}_0(\R)$ be a finitely generated group. Then $G \in \cC$.
\end{cor}
\begin{proof}
To prove the general case when the standard action of $G$ is not necessarily irreducible, we argue by induction on the number $c$ of the connected components of its support $\mathrm{supp}_\mathrm{st}(G)$. This set is finite, since the endpoints of connected components of $\mathrm{supp}_\mathrm{st}(G)$ are included in the set $\bigcup_{s \in S} \partial \mathrm{Fix}(s)$ where $S \subseteq G$ is a finite generating set.

If $c = 1$, then Lemma \ref{prop:irreduciblePL} shows $G \in \cC$ because the group of piecewise projective homeomorphisms of any open interval is isomorphic to $\mathrm{PProj}_0(\R)$. Assume $c \geq 2$, denote by $I \subseteq \R$ the leftmost connected component of $\mathrm{supp}_\mathrm{st}(G)$ and define $J$ as the smallest open interval containing all connected components of $\mathrm{supp}_\mathrm{st}(G)$ except for $I$. Write $N_I$ for the kernel of the action of $G$ on $I$ and $G_I$ for the image of the restriction of $G$ to $I$, and likewise with $N_J, G_J$.

Consider $\varphi \in \mathrm{Rep}_\mathrm{III}(G) \cap \mathrm{Harm}(G)$ and $(t_n)_{n \geq 0} \subseteq \R$ such that $\lim_{n \to \infty} \Psi^{t_n}(\varphi) = \varphi$. If $\varphi(N_I)$ or $\varphi(N_J)$ have global fixed points then they are trivial and we conclude that $\lim_{n \to \infty} t_n = 0$ since $G_I = G/N_I$ and $G_J = G/N_J$ are in $\cC$ by the inductive hypothesis. Thus we may assume that $\varphi(N_I)$ and $\varphi(N_J)$ act irreducibly on $\R$. Notice that $\varphi(N_I)$ cannot be of type III, since it has a non-trivial centralizer $\varphi(N_J)$. We conclude that $\lim_{n \to \infty} t_n  = 0$ by Lemma \ref{lema:noIII}. Again $G \in \cC$ by Theorem \ref{cor:criterio}, finishing the proof. 
\end{proof}

\section{Groups at the doorstep of $\mathcal{C}$} \label{section:noC}
In this section we give an example of a (non-finitely generated) group $H\notin\cC$ which is a direct limit of groups in $\cC$. This shows that $\cC$ is not closed under extensions. We then modify this example to prove Theorem \ref{teo:noc}.

These constructions use the framework of the groups acting on  suspension flows of homeomorphisms of the Cantor set from \cite{MatteBonTriestino2020}, that we briefly recall now.  Let $\varphi\colon X\to X$ be a minimal homeomorphism of the Cantor space. 
Its suspension (or mapping torus) is $Y=(X\times \R)/\Z$, where $\Z$ acts on $X\times \R$ by $n\cdot (x, t)=(\varphi^n(x), t-n)$. We denote by $\pi\colon X\times \R\to Y$ the quotient projection. The suspension flow of $\varphi$ is the flow
\[\Phi\colon Y\times \R\to Y, \quad \quad \Phi^t(\pi(x, s))=\pi(x, s+t).\]

\begin{defn}
We denote by $\Hphi$ the group of all homeomorphisms $f\colon Y\to Y$ such that there exists a continuous function $\tau_f\colon Y\to \R$ with 
\[f(y)=\Phi^{\tau_f(y)}(y)\]
for every $y\in Y$.
\end{defn}
In particular, the group $\Hphi$ preserves every $\Phi$-orbit. Notice that the function $\tau_f$ associated to an element $f$ is unique since $\Phi$ has no periodic orbits, as we are assuming that $\Phi$ is minimal. 

Fix a countable subgroup $G\le \Hphi$. For every point $y\in Y$, we have a representation 
\[\rho_y\colon G\to \mathrm{Homeo}_0(\R)\]
given by 
\begin{equation} \rho_y(g).s=s+\tau_g(\Phi^s(y)). \label{e-rhoy} \end{equation}
In other words, $\rho_y$ is the action read on the $\Phi$-orbit of $y$, identified with $\R$ via the orbital map $\R\to Y, s\mapsto \Phi^s(y)$. It is apparent from \eqref{e-rhoy} that the  representation $\rho_y$ varies continuosly with $y$. It is irreducible provided  $G$ has no fixed point on the orbit of $y$. Recall that we denote by $T_t\colon \R\to \R$ the translation  $s\mapsto s+t$. The identity
\[\rho_{\Phi^t(y)}(g).s= s+\tau_g(\Phi^{s+t}(y))=\rho_y(g).(s+t)-t=T_{-t} \circ \rho_y(g) \circ T_t(s)\]
shows that the map $y\mapsto \rho_y$ intertwines the flow $\Phi$ with the inverse of translation flow $\Psi$  on the space of representations  (see Subsection \ref{sec:harmonic}), namely
\[\rho_{\Phi^t(y)}=\Psi^{-t}(\rho_y).\]
Let $C\subseteq X$ be a clopen subset. If $I\subseteq \R$ is any open interval of length $|I|<1$, then the restriction $\pi_{C, I}$ of the quotient projection to $C\times I$ is a homeomorphism onto an open subset $Y_{C, I}\subseteq Y$. 
We call the map \[\pi_{C, I}\colon C\times I\to Y_{C, I}\]
a \emph{chart}. By abuse of terminology we also call its image $Y_{C, I}$ a chart. Notice that changes of charts are locally of the form $(x, t)\mapsto (\varphi^n(x), T_{-n}(t))$. 

We say that a map $f\colon I\to \R$  is \emph{PL} if it is a {piecewise affine} homeomorphism onto its image, with finitely many discontinuity points for the derivative (henceforth \emph{breakpoints}). We further say that it is \emph{PL dyadic} if the affine maps have slope a power of $2$ and constant term in $\Z[\frac{1}{2}]$, and the breakpoints are in $\Z[\frac{1}{2}]$. We denote by $F_I$ the group of all PL dyadic homeomorphisms $f\colon I\to I$. Thus $F_{(0, 1)}$ is Thompson's group $F$, and $F_I$ is isomorphic to it whenever $J$ has endpoints in $\Z[\frac{1}{2}]$.
\begin{defn}[\cite{MatteBonTriestino2020}]

The group $T(\varphi)$ is the subgroup of $\Hphi$ of all homeomorphisms $g\colon Y\to Y$ such that for any $y\in Y$ and any sufficiently small charts $y\in Y_{C, I}, g(y)\in Y_{D, J}$ such that $g(Y_{C, I})\subseteq Y_{D, J}$, we have
\[\ \pi_{D, J}^{-1} \circ g\circ\pi_{C, I}(x, t)=(\varphi^n(x), f(t))\]
for some $n\in \Z$ and some PL dyadic map $f\colon I\to J$. 
\end{defn}
For every chart $\pi_{C, I}\colon C\times I\to Y_{C, I}$, we denote by $F_{C, I}$ the subgroup of $\Tphi$ obtained by letting $F_I$ act on $C\times I$ diagonally with trivial action on $C$ and transporting this action on $Y_{C, I}$ through the map $\pi_{C, I}$ (and extending it to the trivial action on $Y\setminus Y_{C, I}$). It is shown in \cite[Proposition 4.4]{MatteBonTriestino2020} that the subgroups $F_{C, I}$ generate $\Tphi$.

For a point $y\in Y$, we denote by $\Tphi^0_y$ its \emph{germ stabiliser}, consisting of all $g\in \Tphi$ that  fix pointwise some neighbourhood of $y$.

\begin{prop} \label{prop:nolimites}
Let $\varphi\colon X\to X$ be any minimal homeomorphism of the Cantor space, with suspension $Y$, and fix $z\in Y$. Then the group $H=\Tphi^0_z$ does not belong to $\cC$, but all its finitely generated subgroups belong to $\cC$.

In particular, the class $\cC$ is not closed under direct limits. 
\end{prop}

\begin{proof}
It follows from \cite[Lemma 7.2]{MatteBonTriestino2020} that every finitely generated subgroup of $H$ is isomorphic to a subgroup of a direct product of finitely many copies of Thompson's group $F$, and hence is isomorphic to a subgroup of $F$. Thus every finitely generated subgroup of $H$ belongs to $\cC$ by Theorem \ref{teo:PL}.

To see that $G$ is not itself in $\cC$, choose a point $y\in Y$ not in the $\Phi$-orbit of $z$. Then the corresponding action $\rho_y\colon H\to \mathrm{Homeo}_0(\R)$ given by \eqref{e-rhoy} is irreducible and minimal.  It is also not difficult to see that it is type III (for instance, it follows from \cite[Lemma 7.2]{MatteBonTriestino2020} and from the well-known result of Brin-Squier \cite{BrinSquier1985}, that $H$ is perfect and has no non-abelian free subgroup, hence it cannot have any minimal type I or II action on the line). Since the flow $\Phi$ is minimal, we can choose a sequence $(t_n)_{n \geq 0} \subseteq \R$ with $\lim_{n \to \infty} t_n = \infty$ such that $\lim_{n\to \infty} \Phi^{t_n}(y)=y$. Therefore
\[\rho_y=\lim_{n\to \infty} \rho_{\Phi^{t_n}(y)}=\lim_{n\to \infty} \Psi^{-t_n}(\rho_y),\]
and Theorem \ref{cor:criterio} implies that $H\notin \cC$.
\end{proof}

We now modify this example to prove Theorem \ref{teo:noc}. The idea is to take $G$ to be an extension of (a subgroup of) $H$ by $\Z$, by adding a well-chosen homeomorphism $f$ that fixes the point $z$, and fails to be PL dyadic at the point $z$ only. We first need an elementary lemma.

\begin{lema}\label{lema:sequencePL}
    Let $I=(-\frac{1}{4}, \frac{1}{4})\subset \R$. There exists a sequence of dyadic PL homeomorphisms $f_n\colon I\to I$ with the following properties. 
    \begin{enumerate}
        \item $f_n(x)>x$ for every $x\in I$ and every $n$.
        \item for every $n\ge m$, we have $f_n(x)=f_m(x)$ for every $x\notin (-\frac{1}{4^{m+1}}, \frac{1}{4^{m+1}})$.
        \item $f_n(0)\to 0$ as $n\to \infty$.
    \end{enumerate}
    In particular, $(f_n)_{n \geq 0}$ converges to a (non PL) homeomorphism $f_\infty \colon I \to I$ such that $f_\infty(0)=0$ and  $f_\infty(x)>x$ for every $x\in I\setminus\{0\}$. 
\end{lema}

\begin{proof}
Let $f_1\colon I\to I$ be any dyadic PL homeomorphism such that $f_1(x)>x$ for every $x\in I$, $f_1(-\frac{1}{16})<0$ and $f_1(0)<\frac{1}{16}$. Inductively, suppose to have constructed $f_1,\ldots, f_n$ PL dyadic such that $f_n(x)>x$ for every $x$, $f_m(-\frac{1}{4^{m+1}})<0$, $f_m(0)<\frac{1}{4^{m+1}}$, and $f_n(x)=f_m(x)$  for every $x\notin (-\frac{1}{4^{m+1}}, \frac{1}{4^{m+1}})$ and every $m\ge n$. Then, by standard interpolation arguments for PL maps, we can find $f_{n+1}$ PL dyadic which coincides with $f_n$ on $(-\frac{1}{4}, -\frac{1}{4^{n+1}})\cup(\frac{1}{4^{n+1}}, \frac{1}{4})$ and such that $f_n(x)>x$ for every $x$ and  $f_{n+1}(0)<\frac{1}{4^{n+2}}$. The sequence $(f_n)_{n \geq 0}$ satisfies the desired conclusion. \qedhere
\end{proof}

For the rest of the section, let $I=(-\frac{1}{4}, \frac{1}{4})$ and $(f_n)_{n \geq 0}$ be a sequence as in Lemma \ref{lema:sequencePL}. Fix $x_0\in X$ and set $z=\pi(x_0, 0)\in Y$. Let $(C_n)_{n \geq 0}$ be a sequence of disjoint clopen sets that partition $X\setminus\{x_0\}$. Then the charts $Y_{C_n, I}$ are pairwise disjoint, and \[\bigcup_{n \geq 0} Y_{C_n, I}=Y_{X, I}\setminus\pi(\{x_0\}\times I).\]
It follows that if we define
 $f\colon Y\to Y$  by 
    \begin{equation} f(y)=  \left\{\begin{array}{lr}
      \pi_{X, I}(x, f_n(t))     & \text{if } y=\pi_{X, I} (x, t),  \text{ for }(x, t)\in C_n\times I \\
         \pi_{X, I}(x_0, f_\infty(t))  &  \text{if } y=\pi(x_0, t) \text{ for }t\in I ,\\
        y & \text{if } y\notin Y_{X, I}.
      \end{array}\right. \label{e-elementf} \end{equation}
      Then $f$ is an element of $\Hphi$ supported in $Y_{X, I}$, that fixes $z$, and coincides with some element of $\Tphi$ on a neighbourhood of every point $y\neq z$.
      
Now let $J=(\frac{1}{16}, 1-\frac{1}{16})$. Then $|I\cup J|>1$, and $Y=Y_{X, I}\cup Y_{X, J}$. 
The following result implies Theorem \ref{teo:noc}.
\begin{prop}[Theorem \ref{teo:noc}]
Retain the notations above, and let $G=\langle f, F_{X, J}\rangle \subseteq \Hphi$, where $f$ is given by \eqref{e-elementf}. Then the following hold.
\begin{enumerate}
\item \label{i-Gfg} $G$ is finitely generated.
    \item \label{i-Gsemidirect} $G$ splits as a semi-direct product $G=G^0\rtimes \langle f\rangle$, where $G^0$ is a group whose finitely generated subgroups are all isomorphic to subgroups of Thompson's group $F$. In particular $G$ is amenable if and only if $F$ is amenable.
    
    \item \label{i-GnotinC} $G\notin \cC$.
\end{enumerate}
\end{prop}
\begin{proof}
Recall that the group $F_{X, J}$ is isomorphic to Thompson's group $F$, which is finitely generated, so \eqref{i-Gfg} is obvious. 

The group $G$ fixes $z$.  Let us denote by $G^0$ the normal subgroup consisting of elements that fix a neighbourhood of $z$.
Since $z\notin \overline{Y_{X,J}}$, we have $F_{X, J}\subseteq G^0$. Hence $G/G^0$ is infinite cyclic, generated by the image of $f$, and $G=G^0\rtimes \langle f\rangle$. Every element $g\in G$ coincides with some element of $\Tphi$ on a neighbourhood of every point $y\neq z$, since this  is true for $f$ and for  $F_{X, J}$, and the point $z$ is fixed. It follows that elements of $G^0$ actually coincide with some element of $\Tphi$ on the neighbourhood of every point. Since the definition of $\Tphi$ is local, i.e. membership to $\Tphi$ is defined by a local condition on a neighbourhood of every point,  it follows that $G^0\subseteq \Tphi^0_z$. Hence every finitely generated subgroup of $G^0$ is isomorphic to a subgroup of $F$ by \cite[Lemma 7.2]{MatteBonTriestino2020}, showing \eqref{i-Gsemidirect}.

To show \eqref{i-GnotinC}, fix a point of the form $y=\pi_{X, I}(x, 0)\in Y$, which is not in the $\Phi$-orbit of $z$, and consider the representation $\rho_y\colon G\to \mathrm{Homeo}_0(\R)$. We claim that $\rho_y$ is minimal. Notice that $\rho_y(F_{X, J})$ is a subgroup whose support is the union $\bigcup_{n\in \Z} T_n(J)$ of all integer translates of $J$, and which acts as $F_{T_n(J)}$ on each $T_n(J)$. On the other hand $\rho_y(f)$ is a homeomorphism supported on $\bigcup_{n\in \Z} T_n(I)$, which acts on each $T_n(I)$ as the conjugate $T_n \circ f_m \circ T_{-n}$ of one of the homeomorphisms $f_m$ from Lemma \ref{lema:sequencePL}, for some $m=m(n)$. Since each $f_m$ has no fixed point, and the intervals $T_n(I)$ overlaps with $T_{n-1}(J)$ and $T_n(J)$, it follows that the $\rho_y$-orbit of every $t\in \R$ intersects $T_n(J)$ for every $n$. Hence its closure, that we denote by $C_t$, contains $T_n(J)$ for every $n$, since $\rho_y(F_{X, J})$ acts minimally on each $T_n(J)$. The same reasoning shows that for every $s\in C_t$, there exists $g\in G$ such that $\rho_y(g).s\in J$, thus $s\in \rho_y(g^{-1}).J$ is an interior point of $C_t$. It follows that $C_t$ is open and closed, and thus $C_t=\R$.  This shows that $\rho_y$ is minimal, since $t\in \R$ was arbitrary.

The action $\rho_y$ is clearly not of type I, since it has non-abelian image. It is also not of type II, for instance because $G$ has no non-abelian free subgroup by \eqref{i-Gsemidirect} and by \cite{BrinSquier1985}. Hence $\rho_y$ is of type III. 
We conclude that $G\notin \cC$ as in the proof of Proposition \ref{prop:nolimites}. \qedhere

\end{proof}

\section{Complexity of the semiconjugacy relation} \label{section:complexity}
This section is devoted to the proof of Theorem \ref{teo:complexity}. Subsection \ref{subsection:semiconjugacy} provides the necessary background on preorders on groups and some preliminary lemmas. Subsection \ref{subsection:essentiallyCountable} proves the theorem by reducing semiconjugacy among cocompact actions of a group $G$ to conjugacy on a (Borel) subset of $\mathrm{Rep}_\mathrm{min}(G) \cup \mathrm{Rep}_\mathrm{cyc}(G)$ obtained as the dynamical realization of a certain space of preorders on $G$. The use of this space of preorders also makes apparent the essential countability of all the equivalence relations involved. Subsection \ref{subsection:deroinFinf} gives an example showing that Theorem \ref{teo:existenciaDeroin} on harmonic spaces of finitely generated groups cannot be extended to all non-finitely generated groups.

In this section $G$ is always a countable group. Recall that an irreducible action $\varphi \in \mathrm{Rep}_\mathrm{irr}(G)$ is \emph{cocompact} if there is a compact subset of $\R$ intersecting every $\varphi$-orbit. This condition is equivalent to the existence of a minimal set for $\varphi$, see e.g.~\cite[Lemma 2.1.11]{BMRT}. This in turn is equivalent to the fact that $\varphi$ is semiconjugate to a  minimal or cyclic action in $\mathrm{Rep}_\mathrm{irr}(G)$,  which is unique up to conjugacy. Such an action will be called a \emph{canonical model} for $\varphi$. The space of all cocompact actions of $G$ on $\R$ is denoted $\mathrm{Rep}_\mathrm{cc}(G)$. 

\subsection{Preorders and dynamical realizations} \label{subsection:semiconjugacy}
A \emph{left-preorder} on $G$ is a total, reflexive and transitive binary relation $\preceq$ on $G$ that is also left-invariant, that is, such that $h \preceq k$ implies $gh \preceq gk$ for all $g,h,k \in G$. A left-preorder $\preceq$ is determined by the data of its \emph{positive cone} $P_\preceq = \{g \in G : g \succeq e_G\}$ and its \emph{residue subgroup} $[1]_\preceq = \{g \in G : e_G \preceq g \preceq e_G\}$. We will use the word \emph{preorder} to designate a left-preorder that is not the \emph{trivial preorder} $\preceq_\mathrm{tr}$ with $[1]_{\preceq_\mathrm{tr}} = G$. Denote by $\mathrm{LPO}(G)$ the set of preorders on $G$ with the topology inherited from $\{\preceq, \npreceq\}^{G \times G}$, which becomes a locally compact and totally disconnected Polish space with the induced topology (see \cite[Theorem 2.30]{DecaupRond2019}).

There is a dictionary between preorders and irreducible actions that originates in the proof of the fact that a countable group is left-orderable (that is, admits a preorder where the residue subgroup is trivial) if and only if it embeds into $\mathrm{Homeo}_0(\R)$, see \cite[Theorem 6.8]{Ghys2001}. The interplay between both viewpoints has been exploited to obtain results on the orders of a group \cite{Navas2010} and to study groups acting faithfully on the line \cite{WitteMorris2006}. We refer to \cite{DeroinNavasRivas2016} for a complete treatment of the subject.

We introduce two operations to realize this translation. Given an action $\varphi \in \mathrm{Rep}_\mathrm{irr}(G)$ denote by $\preceq_\varphi \  \in \mathrm{LPO}(G)$ the preorder on $G$ where, for any $g,h \in G$, $g \preceq_\varphi h$ if and only $\varphi(g).0 \leq \varphi(h).0$. Given a preorder $\preceq \  \in \mathrm{LPO}(G)$, following \cite[Theorem 6.8]{Ghys2001} define the \emph{dynamical realization} $\phi_\preceq \ \in \mathrm{Rep}_\mathrm{irr}(G)$ of $\preceq$ as the action of $G$ constructed as follows: fix once and for all a numbering $(g_i)_{i \geq 0}$ of $G$ with $g_0 = e_G$. Define an order-preserving embedding $\iota_\preceq \colon G \to \R$ inductively by setting $\iota_\preceq(g_0) = 0$ and, given the values $\iota_\preceq(g_0),\ldots, \iota_\preceq(g_n)$, set
\[
	\iota_\preceq(g_{n+1}) = \begin{cases} \max\{\iota_\preceq(g_0),\ldots, \iota_\preceq(g_n)\} + 1 & \text{if }g_{n+1} \succ g_0,\ldots, g_n\\
					\min\{\iota_\preceq(g_0),\ldots, \iota_\preceq(g_n)\} - 1 & \text{if }g_{n+1} \prec g_0,\ldots, g_n\\
					(\iota_\preceq(g_i) + \iota_\preceq(g_{j}))/2 & \text{otherwise, }
			\end{cases}
\]
where $g_j = \max\{g_m : 0\leq m\leq n \text{ and } g_m\preceq g_{n+1}\}$ and $g_i = \min\{g_m : 0\leq m\leq n \text{ and } g_m \succeq g_{n+1}\}$. The group $G$ acts on $\iota_\preceq(G)$ by $g.\iota_\preceq(h) = \iota_\preceq(gh)$, and this action extends continuously to $\overline{\iota_\preceq(G)}$. We define $\phi_\preceq$ by extending the action of $G$ on $\overline{\iota_\preceq(G)}$ affinely on its complement in $\R$.

Given a preorder $\preceq$, a subgroup $H \subseteq G$ is said to be \emph{$\preceq$-convex} if whenever $h \in H$ and $g \in G$ are such that $e_G \preceq g \preceq h$, we have $g \in H$. The set of all $\preceq$-convex subgroups is totally ordered, and we denote $H_\preceq$ as the union of all proper $\preceq$-convex subgroups. When $H_\preceq \neq G$ (which is always the case when $G$ is finitely generated, see \cite[Example 2.1.2]{DeroinNavasRivas2016}), we define $\preceq_\ast$ as the preorder on $G$ where $P_{\preceq_\ast} = P_{\preceq} \setminus H_\preceq$ and $[1]_{\preceq_\ast} = H_\preceq$. If $H_\preceq = G$, we set $\preceq_\ast\  =\  \preceq_\mathrm{tr}$. When non-trivial, the preorder $\preceq_\ast$ is called the \emph{minimal model} of $\preceq$ and plays the same role as the canonical model of a cocompact action $\varphi \in \mathrm{Rep}_\mathrm{cc}(G)$. The situation $H_\preceq = G$ is the order-theoretic analogue of the non-existence of a minimal set for an action $\varphi \in \mathrm{Rep}_\mathrm{irr}(G)$.


\begin{prop}[{\cite[Section 14.3.2]{BMRT}}] \label{prop:preordenesBMRT1}
Let $\varphi \in \mathrm{Rep}_\mathrm{irr}(G)$ and $\preceq \ \in \mathrm{LPO}(G)$.

\begin{enumerate}
	\item\label{it:appi} The action $\varphi$ is semiconjugate to the dynamical realization of $\preceq_\varphi$.

	\item\label{it:appii} The dynamical realization of $\preceq$ is cocompact if and only if $\preceq$ has a maximal convex subgroup $H_\preceq \neq G$. In that case, the dynamical realization of $\preceq_\ast$ is a canonical model for the dynamical realization of $\preceq$.

	\item\label{it:appiii} The action $\varphi$ is cocompact if and only if $\preceq_\varphi$ has a minimal model.
\end{enumerate}
\end{prop}

The term \emph{dynamical realization} is usually used in the literature to denote any conjugate of $\phi_\preceq$ (see \cite[Section 1.1.3]{DeroinNavasRivas2016} for instance). The only property that we will need from this definition of $\phi_\preceq$, aside from the fact that $\iota_\preceq$ is explicit, is the following.

\begin{lema} \label{lema:iota}
If $x \in \Z[1/2]$ and $\preceq \  \in \mathrm{LPO}(G)$, then either $x \in \iota_\preceq(G)$ or $x \in \R \setminus \overline{\iota_\preceq(G)}$.
\end{lema}
\begin{proof}
	Say that $y \in \Z[1/2]$ has \emph{height} $n \in \N$ if it can be written as $y = k/2^n$ for some $k \in \Z$. Define $l_n(y)$, the \emph{left neighbor of height $n$} of $y$, as the largest $z \in \Z[1/2]$ with height $n$ such that $z \leq y$. Define the right neighbor $r_n(y)$ of height $n$ of $y$ similarly.
\begin{claim}
	If $x \in \iota_\preceq(G)$ has height $n$, then $l_k(x)$ and $r_k(x)$ belong to $\iota_\preceq(G)$ for every $0 \leq k < n$.
\end{claim}
\begin{proof}[Proof of the claim]
	We argue that this holds true by induction at every stage of the construction of $\iota_\preceq(G)$. We may assume that $x$ does not have height $k$ for any $0 \leq k < n$, since $l_j(x) = r_j(x) = x$ whenever $x$ has height $j$. Thus if $x \in \Z$ then $n = 0$ and the statement follows, so we may assume that $n > 0$ and that $x$ was added as a midpoint of elements $y_1, y_2 \in \iota_\preceq(G)$ with $y_1 < x < y_2$. Then $l_{k}(x) = l_k(y_1),\, r_{k}(x) = r_k(y_2)$ for every $0 \leq k < n$, and $l_k(y_1), r_k(y_2) \in \iota_\preceq(G)$ by the inductive hypothesis.
\end{proof}
Now suppose that $x \in \Z[1/2]$ belongs to $\overline{\iota_\preceq(G)}$. Write $x = k/2^j$ for some $k \in \Z, j \in \N$. Let $(x_n)_{n \geq 0} \subseteq \iota_\preceq(G)$ be a sequence converging to $x$, all of whose elements may be assumed to be different from $x$. Consider $x_n \in ((k-1)/2^j, (k+1)/2^j)$ such that $x_n$ does not have height $j$, so either $r_j(x_n) = x$ or $l_j(x_n) = x$. Hence $x \in \iota_\preceq(G)$ by the previous claim.
\end{proof}

By identifying a subgroup of $G$ with its indicator function, the set $\mathrm{Sub}(G)$ of all subgroups of $G$ becomes a subset of $\{0,1\}^G$ and inherits a compact metrizable topology \cite{Chabauty1950}. Its Borel structure is generated by the sets $\{H \in \mathrm{Sub}(G) : g \in H\}$ where $g \in G$. A subgroup $H \subseteq G$ is \emph{proper} if $H \neq G$.

We now prove that the function assigning to a preorder its maximal convex subgroup is Borel.

\begin{lema} \label{lema:convexoMax}
	The map $\preceq \  \in \mathrm{LPO}(G) \mapsto H_\preceq \in \mathrm{Sub}(G)$ is Borel.
\end{lema}
\begin{proof}
Fix $g \in G$ and consider the Borel set
\[
	B_g = \{(\preceq, H) \in \mathrm{LPO}(G) \times \mathrm{Sub}(G) : \text{$H$ is proper, $\preceq$-convex and contains $g$}\}.
\]
Let $\pi \colon \mathrm{LPO}(G) \times \mathrm{Sub}(G) \to \mathrm{LPO}(G)$ be the projection onto the first coordinate, and consider a section $S = B_g \cap \pi^{-1}(\preceq)$ for some fixed preorder $\preceq$. Then $S$ is homeomorphic to
\[
	\{H \in \mathrm{Sub}(G) : \text{$H$ is proper, $\preceq$-convex and contains $g$}\},
\]
which can be written as
\begin{equation} \label{eq:ksigma}
	\bigcup_{u \in G} \bigcap_{v \succeq e_G}\left( \{ H : u \not \in H\text{ and }v,g \in H\} \cup \{H : \text{$u \not \in H,$\,  $g \in H$, and $h \preceq v$ for all $h \in H$}\} \right).
\end{equation}
Indeed, a subgroup $H \subseteq G$ is proper and $\preceq$-convex if and only if there exists $u \in G \setminus H$, and for every $v \succeq e_G$ either $v \in H$ or $h \preceq v$ for all $h \in H$. The expression \eqref{eq:ksigma} shows that $S$ is $K_\sigma$ (that is, a countable union of compact sets), so Theorem \ref{teo:ksigma} implies that $\pi(B_g)$ is Borel and that there is a Borel map $\zeta_g \colon \mathrm{LPO}(G) \to \mathrm{Sub}(G)$ such that $(\preceq, \zeta_g(\preceq)) \in B_g$ if $\preceq \ \in \pi(B_g)$ and $\zeta_g(\preceq) = \{e_G\}$ if not.

The condition $\preceq \ \in \pi(B_g)$ is equivalent to the existence of a proper, $\preceq$-convex subgroup $H$ containing $g$. Thus we may write $H_\preceq = \bigcup_{g \in G}\zeta_g(\preceq)$ for every $\preceq \ \in \mathrm{LPO}(G)$. We conclude that $\preceq \ \mapsto H_\preceq$ is Borel too.
\end{proof}

\subsection{Proof of Theorem \ref{teo:complexity}} \label{subsection:essentiallyCountable}
Denote by $\mathrm{LPO}_\mathrm{can}(G)$ the set of preorders $\preceq$ on $G$ such that $H_\preceq = [1]_\preceq$. From Lemma \ref{lema:convexoMax} it is clear that $\mathrm{LPO}_\mathrm{can}(G) \subseteq \mathrm{LPO}(G)$ is Borel. To study semiconjugacy among cocompact actions (to show its essential countability, specifically) we introduce the equivalence relation $\cR$ on $\mathrm{LPO}_\mathrm{can}(G)$ by declaring $\preceq \cR \preceq'$ if the dynamical realizations $\iota_\preceq, \iota_{\preceq'} \in \mathrm{Rep}_\mathrm{min}(G) \cup \mathrm{Rep}_\mathrm{cyc}(G)$ are conjugate.

\begin{prop}  \label{prop:reduccion}
The relation of semiconjugacy on $\mathrm{Rep}_\mathrm{cc}(G)$ is Borel reducible to $\cR$, which is in turn Borel reducible to conjugacy on $\mathrm{Rep}_\mathrm{min}(G) \cup \mathrm{Rep}_\mathrm{cyc}(G)$ (and all the spaces involved are standard Borel).
\end{prop}
\begin{proof}
By writing
\[
	\mathrm{Rep}_\mathrm{irr}(G) = \bigcap_{n \in \N} \bigcup_{g,h \in G} \{ \varphi \colon G \to \mathrm{Homeo}_0(\R) : \varphi(g).0> n,\, \varphi(h).0<-n\},
\]
\begin{align*}
	\mathrm{Rep}_\mathrm{cc}(G)  & = \{\varphi \in \mathrm{Rep}_\mathrm{irr}(G) : \text{there is }n\geq 1 \text{ such that } \varphi(G).[-n,n] \text{ covers }\R\} \\
	 & = \bigcup_{n \geq 1} \bigcap_{k \in \Z} \bigcup_{\substack{F \subseteq G\\ \text{ finite}} } \{\varphi \in \mathrm{Rep}_\mathrm{irr}(G) : \varphi(F).[-n,n] \supseteq [k,k+1]\}
\end{align*}
and
\[
	\mathrm{Rep}_\mathrm{cyc}(G) = \bigcup_{h \in G} \bigcap_{g \in G} \bigcup_{n \in \Z}\{ \varphi \in \mathrm{Rep}_\mathrm{irr}(G) : \varphi(g) = \varphi(h)^n\}
\]
we see that $\mathrm{Rep}_\mathrm{irr}(G)$ is $G_\delta$ in $\mathrm{Homeo}_0(\R)^G$ and that $\mathrm{Rep}_\mathrm{cc}(G), \mathrm{Rep}_\mathrm{cyc}(G) \subseteq \mathrm{Rep}_\mathrm{irr}(G)$ are Borel.

Consider the maps
\[
	\eta_1 \colon \varphi \in \mathrm{Rep}_\mathrm{cc}(G) \mapsto (\preceq_\varphi)_\ast \in \mathrm{LPO}_\mathrm{can}(G) \]
and
\[
	\eta_2 \colon \preceq \ \in \mathrm{LPO}_\mathrm{can}(G) \mapsto \phi_\preceq \in \mathrm{Rep}_\mathrm{min}(G) \cup \mathrm{Rep}_\mathrm{cyc}(G).
\]
Proposition \ref{prop:preordenesBMRT1}, \eqref{it:appiii} (resp. \eqref{it:appii}) implies that $\eta_1$ (resp. $\eta_2$) is well defined.

\setcounter{claimnum}{0}
\begin{claimnum}
$\eta_1$ is a Borel map.
\end{claimnum}
\begin{proof}[Proof of the claim]
Write $\eta_1$ as the composition of 
\[
	\varphi \in \mathrm{Rep}_\mathrm{cc}(G) \mapsto \ \preceq_\varphi \  \in \mathrm{LPO}(G) \quad \text{ and }\quad \preceq \ \in \mathrm{LPO}(G) \mapsto \ \preceq_\ast \ \in \mathrm{LPO}(G) \cup \{\preceq_\mathrm{tr}\}.
\]
The first map is Borel, since
\[
	\{\varphi \in \mathrm{Rep}_\mathrm{cc}(G) : e_G \preceq_\varphi g\} = \{\varphi \in \mathrm{Rep}_\mathrm{cc}(G) : \varphi(g).0 \geq 0\}
\]
for any $g \in G$. To see that the second map is Borel, notice that for any $g \in G$ we have
\[
	\{\preceq \ \in \mathrm{LPO}(G) : e_G \preceq_\ast g\} = \{\preceq \ \in \mathrm{LPO}(G) : e_G \preceq g \text{ or }g \in H_\preceq\}
\]
where $H_\preceq \subseteq G$ is the maximal convex subgroup of $\preceq$. Since $\preceq \ \mapsto H_\preceq$ is Borel by Lemma \ref{lema:convexoMax} we conclude that $\eta_1$ is also Borel.
\end{proof}

\begin{claimnum}
$\eta_2$ is a Borel map.
\end{claimnum}
\begin{proof}[Proof of the claim]
Let $x \in \R$ and fix $g \in G$. We want to prove that the set 
\[
	\{\preceq \ \in \mathrm{LPO}(G) : \phi_\preceq(g).x > x\}
\]
is Borel, and by density of $\Z[1/2]$ in $\R$ we may assume that $x \in \Z[1/2]$. By Lemma \ref{lema:iota}, we have that for every $\preceq \ \in \mathrm{LPO}(G)$ either $x \in \iota_\preceq(G)$ or $x \in \R \setminus \overline{\iota_\preceq(G)}$. Thus
\begin{align} \label{eq:borelSet}
\{\preceq \ \in \mathrm{LPO}(G) : \phi_\preceq(g).x > x\} & = \bigcup_{h \in G}\{\preceq \  : \iota_\preceq(h) = x \text{ and } gh \succ h\} \\
	& \cup \bigcup_{h_1, h_2 \in G} \{\preceq \ : \iota_\preceq(h_1) < x < \iota_\preceq(h_2),\, gh_1 \succ h_1 \nonumber \\  & \text{ where } h_1 = \max\{h \in G : h\prec h_2\}, h_2 = \min\{h \in G : h\succ h_1\} \}. \nonumber
\end{align}
The explicit construction of $\iota_\preceq$ from $\preceq$ implies that all the sets $\{\preceq \ : \iota_\preceq(h) = y\}$ where $h \in G$, $y \in \Z[1/2]$ are Borel, and we conclude that the set \eqref{eq:borelSet} is also Borel.
\end{proof}
Now let $\varphi \in \mathrm{Rep}_\mathrm{cc}(G)$. By definition we have that $\varphi$ is semiconjugate to $\phi_{\preceq_\varphi}$, whose canonical model is $\phi_{(\preceq_\varphi)_\ast} = \eta_2 \circ \eta_1(\varphi)$ by Proposition \ref{prop:preordenesBMRT1} \eqref{it:appii}. This implies that $\eta_1$ is a reduction of semiconjugacy on $\mathrm{Rep}_\mathrm{cc}(G)$ to $\cR$. By its very definition, $\cR$ is reducible to conjugacy on $\mathrm{Rep}_\mathrm{min}(G) \cup \mathrm{Rep}_\mathrm{cyc}(G)$.
\end{proof}

\begin{rmk}
 We will not need this, but the relation $\cR$ can be defined without fixing a particular definition of the dynamical realization of a preorder: $\cR$ coincides with the relation obtained by declaring $\preceq, \preceq' \  \in \mathrm{LPO}_\mathrm{can}(G)$ to be related if $G$ admits an order-preserving action on a preordered set $(\cX, <)$ on which the preordered sets $(G, \preceq)$ and $(G, \preceq')$ embed $G$-equivariantly and such that the images of $G$ under each embedding are unbounded in both directions for $<$. Here, a \emph{preordered set} is a set endowed with a total, reflexive and transitive binary relation.
\end{rmk}

\begin{prop} \label{prop:Rcountable}
The relation $\cR$ is essentially countable.
\end{prop}
\begin{proof}
Let $\cF \subseteq \mathrm{LPO}_\mathrm{can}(G)$ be the set of preorders $\preceq$ such that $\phi_\preceq$ is not of type I. Since the actions $\phi_\preceq$ are canonical models when $\preceq \in \mathrm{LPO}_\mathrm{can}(G)$, the set $\mathrm{LPO}_\mathrm{can}(G) \setminus \cF$ can be written as
\[
	\{ \preceq \ \in \mathrm{LPO}_\mathrm{can}(G) : \text{for every }g \in G, \mathrm{Fix}(\phi_\preceq(g)) = \emptyset \text{ or } \mathrm{Fix}(\phi_\preceq(g)) = \mathrm{id}_\R\}
\]
so $\cF$ is Borel by (the proof of) Proposition \ref{prop:reduccion}. 

It is clear that $\cF$ is $\cR$-invariant.

\begin{claim}
The relation $\cR$ restricted to $\mathrm{LPO}_\mathrm{can}(G) \setminus \cF$ is smooth.
\end{claim}
\begin{proof}[Proof of the claim]
By the definition of $\cF$, the relation $\cR$ restricted to $\mathrm{LPO}_\mathrm{can}(G) \setminus \cF$ Borel reduces to conjugacy in $\mathrm{Rep}_\mathrm{min}(G/[G,G]) \cup \mathrm{Rep}_\mathrm{cyc}(G/[G,G])$. Corollary \ref{cor:abeliano} shows that conjugacy on $\mathrm{Rep}_\mathrm{min}(G/[G,G])$ is smooth. The same is true for conjugacy on $\mathrm{Rep}_\mathrm{cyc}(G)$ (and \emph{a fortiori} on $\mathrm{Rep}_\mathrm{cyc}(G/[G,G])$): given $\varphi \in \mathrm{Rep}_\mathrm{cyc}(G)$ there exists a unique isomorphism between $\varphi(G) \subseteq \mathrm{Homeo}_0(\R)$ and $\Z$ sending the generator $f$ of $\varphi(G)$ with $f(0)> 0$ to $1 \in \Z$. We obtain a uniquely defined morphism $\pi_\varphi \colon G \to \Z$, and the assignment $\varphi \in \mathrm{Rep}_\mathrm{cyc}(G) \mapsto \pi_\varphi \in \mathrm{Hom}(G, \Z)$ is a Borel reduction of the conjugacy relation on $\mathrm{Rep}_\mathrm{cyc}(G)$ to equality on the (standard Borel) space $\mathrm{Hom}(G, \Z)$.
\end{proof}

Thus we may restrict our considerations to $\cF$. We will show that the subset
\[
	T = \{\preceq \ \in \cF : \text{there exists }g \in G \text{ such that }0 \in \partial \mathrm{Fix}(\phi_\preceq(g))\}.
\]
is a \emph{complete countable Borel section} for $\cR$ on $\cF$, that is, $T$ is Borel and intersects any $\cR$-class $[\preceq]$ in $\cF$ in a countable non-empty set. An application of the Lusin-Novikov theorem shows that $\cR$ restricted to $\cF$ is essentially countable (see \cite[Corollary 7.5.3]{Gao2009}, for instance).

To see that $T$ is Borel, notice that a preorder $\preceq \ \in \cF$ is in $T$ if and only if there is $g \in [1]_\preceq$ such that either:
\begin{itemize}
	\item there is $x \succ e_G$ such that for every $y \in G$ with $y \succ e_G$, there is $n \in \N$ with $g^nx \prec y$, or
	\item there is $x \prec e_G$ such that for every $y \in G$ with $y \prec e_G$, there is $n \in \N$ with $g^nx \succ y$.
\end{itemize}
This description readily implies that $T$ is Borel in $\cF$. 

Now fix $[\preceq]$ an $\cR$-class in $\cF$. Since $\phi_\preceq$ is not of type I there is $g \in G$ such that $\mathrm{Fix}(\phi_\preceq(g))$ is not $\R$ nor $\emptyset$. Choose $x \in \partial \mathrm{Fix}(\phi_\preceq(g))$, define $\psi = \Psi^{-x}(\phi_\preceq)$ and let $\preceq' \ = \ \preceq_\psi$ be the left preorder induced by the $\psi(G)$-orbit of 0. By Proposition \ref{prop:preordenesBMRT1}, \eqref{it:appii}, we have $\preceq' \ \in \mathrm{LPO}_\mathrm{can}(G)$ so $\preceq' \ \in T$. Then $\phi_\preceq$, $\psi$ and $\phi_{\preceq'}$ are all conjugate, hence $\preceq' \ \in T \cap [\preceq]$ and $T \cap [\preceq]$ is non-empty.

To see that $T \cap [\preceq]$ is at most countable, fo each $\preceq' \ \in T \cap [\preceq]$ consider $f \in \mathrm{Homeo}_0(\R)$ such that $\phi_\preceq = f \circ \phi_{\preceq'} \circ f^{-1}$. Notice that $f(0)$ belongs to the countable set $\bigcup_{g \in G} \partial \mathrm{Fix}(\phi_{\preceq}(g))$ because $\preceq' \ \in T$. For any $g \in G$ we have the equivalences
\[
	g \preceq' e_G \text{ if and only if }\phi_{\preceq'}(g).0 > 0 \text{ if and only if } \phi_\preceq(g).f(0) > f(0),
\]
showing that $\preceq'$ is the preorder induced from the $\phi_\preceq(G)$-orbit of $f(0)$. Thus there are countably many possibilities for $\preceq'$.
\end{proof}

We may now conclude the proof of Theorem \ref{teo:complexity}, which we restate for convenience.

\begin{cor}[Theorem \ref{teo:complexity}]
	Let $G$ be a countable group. The semiconjugacy relation between cocompact actions of $G$ on the line is essentially countable, and it is smooth if and only if $G \in \cC$.
\end{cor}
\begin{proof}
Semiconjugacy on $\mathrm{Rep}_\mathrm{cc}(G)$ is Borel reducible to $\cR$ by the first statement of Proposition \ref{prop:reduccion}, and both are essentially countable by Proposition \ref{prop:Rcountable}.

Conjugacy on $\mathrm{Rep}_\mathrm{min}(G) \cup \mathrm{Rep}_\mathrm{cyc}(G)$ is clearly Borel reducible to semiconjugacy on $\mathrm{Rep}_\mathrm{cc}(G)$, and Proposition \ref{prop:reduccion} shows that both relations are bireducible to each other. The proof of Proposition \ref{prop:Rcountable} shows that conjugacy on $\mathrm{Rep}_\mathrm{cyc}(G)$ is always smooth, so semiconjugacy on $\mathrm{Rep}_\mathrm{cc}(G)$ is smooth if and only if $G \in \cC$.
\end{proof}

\begin{rmk}
 Recall that two cocompact actions $\varphi_1, \varphi_2 \in \mathrm{Rep}_\mathrm{cc}(G)$ are \emph{pointed semiconjugate} if there is an action $\eta \in \mathrm{Rep}_\mathrm{cc}(G)$ that is minimal or cyclic, and semiconjugacies $h_i \colon \R \to \R$ between $\varphi_i$ and $\eta$ for $i = 1,2$ such that $h_1(0) = h_2(0)$. It is not hard to see that the map $\eta_1$ from the proof of Proposition \ref{prop:reduccion} exhibits a Borel reduction of pointed semiconjugacy on $\mathrm{Rep}_\mathrm{cc}(G)$ to equality on $\mathrm{LPO}_\mathrm{can}(G)$. Therefore the relation of pointed semiconjugacy is always smooth.
\end{rmk}
\appendix
\section{There is no analogue of the space of harmonic actions for infinitely generated groups} \label{subsection:deroinFinf}
The finite generation assumption in Theorem \ref{teo:existenciaDeroin} comes from the fact that its proof is based on the study of random walk on $G$ driven by symmetric  measures supported on a finite generating set \cite{DKNP}. While this approach has many advantages, if one is only interested in proving the existence of a retract $\mathrm{Harm}(G)\subseteq \mathrm{Rep}_\mathrm{irr}(G)$ satisfying the conclusion of Theorem \ref{teo:existenciaDeroin}, it could be reasonable to expect that a more elementary proof could be given without using random walks and avoiding finite generation (possibly after relaxing compactness of $\mathrm{Harm}(G)$ to local compactness). 

However we show here that when $G$ is not finitely generated, there is in general no {continuous} retraction of $\mathrm{Rep}_\mathrm{irr}(G)$ onto a set or representatives of pointed semiconjugacy classes.

Denote by $\Gamma$ the group $\{f \in \mathrm{Homeo}_0(\R) : f(0) = 0\}$, so two minimal actions $\varphi, \psi \in \mathrm{Rep}_\mathrm{min}(G)$ are pointed semiconjugate if and only if they are conjugate by an element of $\Gamma$.
\begin{lema}\label{lema:Finfty}
Let $G = F_\infty$ be the free group on an infinite countable set and let $U,V \subseteq \mathrm{Rep}_\mathrm{min}(G)$ be non-empty open subsets that are invariant under conjugation by $\Gamma$. Then $\overline{U} \cap \overline{V} \neq \emptyset$.
\end{lema}
\begin{proof}
Let $\cS$ be an infinite free generating set of $G$ and let $S \subset \cS$ be a finite subset such that $U$ contains an non-empty open set of the form
\[
	W = \{ \varphi \in \mathrm{Rep}_\mathrm{min}(G) : \varphi(s).(x_{i,s}) \in I_{i,s} \text{ for all }s \in S \text{ and }i = 1,\ldots, m\}
\]
where $x_{i,s} \in \R$ and $I_{i,s} \subset \R$ are open intervals for $i = 1, \ldots, m$ and $s \in S$. Let $\psi \in \mathrm{Rep}_\mathrm{min}(G)$ be such that $\psi(s) = \mathrm{id}_\R$ for all $s \in S$. Such an action exists because $G$ has infinite rank. We will show that $\psi \in \overline{U}$ by using a variant of the ``Alexander trick'', and since $\psi$ was arbitrary this will imply the claim.

Take $M > 0$ such that every $x_{i,s}, I_{i,s}$ is included in $[-M + 1,M - 1]$ and fix $\varphi \in W$. Let $n \in \N_+$ and choose $f_n \in \Gamma$ such that $\abs{f_n^{-1}(M)}, \abs{f_n^{-1}(-M)} < 1/n$. Define an action $\psi_n$ by setting $\psi_n(s) = f_n^{-1} \circ \psi(s) \circ f_n$ if $s \in \cS \setminus S$ and imposing
\[
	\restr{\psi_n(s)}{[-M + 1, M-1]} = \restr{\varphi(s)}{[-M + 1,M - 1]} \text{ and } \restr{\psi_n(s)}{\R \setminus [-M, M]}  = \mathrm{id}_{\R \setminus [-M,M]}
\]
for all $s \in S$. Then $\psi_n$ is minimal since $f_n^{-1} \circ \psi \circ f_n$ is minimal and $\psi_n \in W$ by construction.

We have $\lim_{n \to \infty} f_n. \psi_n = \psi$ and $f_n. \psi_n \in f_n. W \subseteq f_n.U = U$, so $\psi \in \overline{U}$ as desired. \qedhere
\end{proof}

\begin{prop}
There is no continuous map $r$ from $\mathrm{Rep}_\mathrm{min}(F_\infty)$ to a regular (e.g. metric) topological space $X$ such that the fibers of $r$ are the pointed semiconjugacy classes.
\end{prop}

\begin{proof}
If there were such a map $r$, take two non-conjugate actions $\varphi, \psi \in \mathrm{Rep}_\mathrm{min}(F_\infty)$ and let $U, V \subset X$ be two disjoint open neighborhoods of $r(\varphi), r(\psi)$ respectively such that $\overline{U} \cap \overline{V} = \emptyset$. Then $r^{-1}(U), r^{-1}(V) \subseteq \mathrm{Rep}_\mathrm{min}(F_\infty)$ would be non-empty open sets that are invariant by $\Gamma$ and with disjoint closures, contradicting the previous lemma.
\qedhere
\end{proof}

\providecommand{\bysame}{\leavevmode\hbox to3em{\hrulefill}\thinspace}
\providecommand{\MR}{\relax\ifhmode\unskip\space\fi MR }
\providecommand{\MRhref}[2]{%
  \href{http://www.ams.org/mathscinet-getitem?mr=#1}{#2}
}
\providecommand{\href}[2]{#2}


\typeout{get arXiv to do 4 passes: Label(s) may have changed. Rerun}
\end{document}